\documentclass[11pt]{article}

\usepackage[USenglish]{babel}
\usepackage[T1]{fontenc} 
\usepackage[utf8,latin1]{inputenc}

\usepackage{amsmath,amsthm,amssymb,amsfonts}
\usepackage{dsfont}
\usepackage{mathtools}
\usepackage{mathrsfs}
\usepackage{tensor} 

\usepackage{authblk}						

\usepackage[font={small}]{caption}

\mathtoolsset{showonlyrefs,showmanualtags}  

\usepackage{enumitem}
\usepackage{lmodern}

\usepackage[bottom=3cm, top=3cm, left=3.5cm, right=3.5cm]{geometry}

\usepackage{subfigure} 

\usepackage[bookmarks=true]{hyperref}
\usepackage{xcolor}
\hypersetup{
    colorlinks,
    linkcolor={red!50!black},
    citecolor={blue!50!black},
    urlcolor={blue!80!black},
}

\newcommand{\N}{\mathbb{N}}
\newcommand{\R}{\mathbb{R}}

\newcommand{\de}{\hspace{0.06em} d}

\newcommand{\sr}{sub-Riemannian }

\newcommand{\D}{\mathcal{D}}
\newcommand{\G}{\mathcal{G}}

\newcommand{\area}{\sigma}
\newcommand{\diverg}{\mathrm{div}_\omega}
\newcommand{\deltasr}{\Delta}

\newcommand{\deltas}{\delta}
\newcommand{\smallpar}{r_0}
\newcommand{\eps}{\varepsilon}

\newcommand{\omNice}[1]{\Omega_{#1}}
\newcommand{\omprimeNice}[1]{\Omega_0^{#1}}
\newcommand{\bdomNice}[1]{\partial\omNice{#1}}

\newcommand{\omext}[1]{\Omega^{-#1}}

\newcommand{\extfunsp}{C_c^\infty(\Omega_{-\smallpar}^{\smallpar})}

\newcommand{\op}{I^c}

\theoremstyle{plain}
\newtheorem{thm}{Theorem}[section]
\newtheorem*{thm*}{Theorem}
\newtheorem{cor}[thm]{Corollary}
\newtheorem{lem}[thm]{Lemma}
\newtheorem*{lem*}{Lemma}
\newtheorem{prop}[thm]{Proposition}

\theoremstyle{definition}
\newtheorem{defn}[thm]{Definition}

\theoremstyle{remark}
\newtheorem{rmk}[thm]{Remark}

\newcommand{\hiddensubsection}[1]{
    \stepcounter{subsection}
    \subsection*{\arabic{section}.\arabic{subsection}\hspace{1em}{#1}}
}

\title{Relative heat content asymptotics for sub-Riemannian manifolds}
\date{\today}
\author[1]{Andrei Agrachev}
\author[1,2]{Luca Rizzi}
\author[1,2,3]{Tommaso Rossi}
\affil[1]{SISSA, Trieste, Italy}
\affil[2]{Univ. Grenoble Alpes, CNRS, Institut Fourier, Grenoble, France}
\affil[3]{Institut f\"ur Angewandte Mathematik, Universit\"at Bonn, Bonn, Germany}

\begin{document}
\maketitle 

\begin{abstract}
The relative heat content associated with a subset $\Omega\subset M$ of a \sr manifold, is defined as the total amount of heat contained in $\Omega$ at time $t$, with uniform initial condition on $\Omega$, allowing the heat to flow outside the domain. In this work, we obtain a fourth-order asymptotic expansion in square root of $t$ of the relative heat content associated with relatively compact non-characteristic domains. Compared to the classical heat content that we studied in \cite{MR4223354}, several difficulties emerge due to the absence of Dirichlet conditions at the boundary of the domain. To overcome this lack of information, we combine a rough asymptotics for the temperature function at the boundary, coupled with stochastic completeness of the heat semi-group. Our technique applies to any (possibly rank-varying) sub-Riemannian manifold that is globally doubling and satisfies a global weak Poincar\'e inequality, including in particular sub-Riemannian structures on compact manifolds and Carnot groups.
\end{abstract}


\tableofcontents

\section{Introduction} 
In this paper we study the asymptotics of the relative heat content in sub-Riemannian geometry. The latter is a vast generalization of Riemannian geometry, indeed a \sr manifold $M$ is a smooth manifold where a metric is defined only on a subset of preferred directions $\D_x\subset T_xM$ at each point $x\in M$ (called horizontal directions). For example, $\D$ can be a sub-bundle of the tangent bundle, but we will consider the most general case of rank-varying distributions. Moreover, we assume that $\D$ satisfies the so-called H\"ormander condition, which ensures that $M$ is horizontally-path connected, and that the usual length-minimization procedure yields a well-defined metric. 

Let $M$ be a \sr manifold, equipped with a smooth measure $\omega$, let $\Omega\subset M$ be an open relatively compact subset of $M$, with smooth boundary, and consider the Cauchy problem for the heat equation in this setting: 
\begin{equation}\label{eqn:cauchy_prob}
\begin{aligned}
\left(\partial_t -\Delta\right)u(t,x)  & =  0, & \qquad &\forall (t,x) \in (0,\infty) \times M, \\
u(0,\cdot) & 	=  \mathds{1}_\Omega,  & \qquad &\text{in }L^2(M,\omega),
\end{aligned}
\end{equation}
where $\mathds{1}_\Omega$ is the indicator function of the set $\Omega$, and $\Delta$ is the sub-Laplacian, defined with respect to $\omega$. By classical spectral theory, there exists a unique solution to \eqref{eqn:cauchy_prob},
\begin{equation}
u(t,x)=e^{t\Delta}\mathds{1}_\Omega(x), \qquad\forall\,x\in M,\ t>0,
\end{equation}
where $e^{t\Delta}$ denotes the heat semi-group in $L^2(M,\omega)$, associated with $\Delta$. The {\em relative heat content} is the function
\begin{equation}
H_\Omega(t)=\int_\Omega u(t,x)d \omega(x),\qquad\forall\,t>0.
\end{equation}

This quantity has been studied in connection with geometric properties of subsets of $\R^n$, starting from the seminal work of De Giorgi \cite{MR62214}, where he introduced the notion of perimeter of a set in $\R^n$ and proved a characterization of sets of finite perimeter in terms of the heat kernel. His result was subsequently refined, using techniques of functions of bounded variation: it was proven in \cite{MR1309086} for balls in $\R^n$, and in \cite{MR2325595} for general subsets of $\R^n$,  that a borel set $\Omega\subset\R^n$ with finite Lebesgue measure has finite perimeter \`a la De Giorgi if and only if 
\begin{equation}
\label{eq:de_giorgi}
\exists\lim_{t\to 0}\frac{\sqrt{\pi}}{\sqrt{t}}\big(|\Omega|-H_\Omega(t)\big)=P(\Omega),
\end{equation} 
where $|\cdot|$ is the Lebesgue measure and $P$ is the perimeter measure in $\R^n$. Notice that \eqref{eq:de_giorgi} is equivalent to a first-order\footnote{Here and throughout the paper, the notion of order is computed with respect to $\sqrt{t}$.} asymptotic expansion of $H_\Omega(t)$. A further development in this direction was then obtained in \cite{MR3019137}, where the authors extended \eqref{eq:de_giorgi} to an asymptotic expansion of order $3$ in $\sqrt{t}$, assuming the boundary of $\Omega\subset\R^n$ to be a $C^{1,1}$ set. For simplicity, we state here the result of \cite[Thm.\ 1.1]{MR3019137} assuming $\partial\Omega$ is smooth\footnote{The statement of Theorem 1.1 in \cite{MR3019137} differs from \eqref{eq:eucl_exp} by a sign in the third-order coefficient: the correct sign appears a few lines below the statement, in the expansion of the function $K_{t}(E,E^c)$.
}:
\begin{multline}
\label{eq:eucl_exp}
H_\Omega(t)=|\Omega|-\frac{1}{\sqrt\pi}P(\Omega)t^{1/2}
\\
+\frac{(n-1)^2}{12\sqrt\pi}\int_{\partial\Omega}\left(H_{\partial \Omega}^2(x)+\frac{2}{(n-1)^2}c_{\partial \Omega}(x)\right)d\mathcal{H}^{n-1}(x)t^{3/2}+o(t^{3/2}),
\end{multline} 
as $t\to0$, where $\mathcal{H}^{n-1}$ is the Hausdorff measure and, denoting by $k_i^{\partial \Omega}(x)$ the principal curvatures of $\partial \Omega$ at the point $x$, 
\begin{equation}
H_{\partial \Omega}(x)=\frac{1}{n-1}\sum_{i=1}^{n-1}k_i^{\partial \Omega}(x), \qquad c_{\partial \Omega}(x)=\sum_{i=1}^{n-1}k_i^{\partial \Omega}(x)^2,
\end{equation}

In the Riemannian setting, Van den Berg and Gilkey in \cite{MR3358065} proved the existence of a complete asympotic expansion for $H_\Omega(t)$, generalizing \eqref{eq:eucl_exp}, when $\partial\Omega$ is smooth. Moreover, they were able to compute explicitly the coefficients of the expansion up to order $4$ in $\sqrt{t}$. Their techniques are based on pseudo-differential calculus, and cannot be immediately adapted to the sub-Riemannian setting. In particular, what is missing is a global parametrix estimate for the heat kernel $p_t(x,y)$, cf.\ \cite[Sec.\ 2.3]{MR3358065}: for any $k\in\N$, there exist $J_k,C_k>0$ such that
\begin{equation}
\label{eq:parametrix}
\bigg\|p_t(x,y)-\sum_{j=0}^{J_k}p_t^j(x,y)\bigg\|_{C^k(M\times M)}\leq C_kt^k,\qquad\text{as }t\to0,
\end{equation}
where $p_t^j(x,y)$ are suitable smooth functions, given explicitly in terms of the Euclidean heat kernel and iterated convolutions. The closest estimate analogue to \eqref{eq:parametrix} in the \sr setting is the one proved recently in \cite[Thm.\ A]{YHT-2} (see Theorem \ref{t:hk_exp} for the precise statement), where the authors show an asymptotic expansion of the heat kernel in an \emph{asymptotic neighborhood} of the diagonal, which is not enough to reproduce \eqref{eq:parametrix} and thus the argument of Van den Berg and Gilkey. Moreover, in this case, $p_t^j(x,y)$ is expressed in terms of the heat kernel of the nilpotent approximation and iterated convolutions, thus posing technical difficulties for the explicit computations of the coefficients (which would be no longer ``simple'' gaussian-type integrals). 

In this paper, under the assumption of not having characteristic points, we prove the existence of the asymptotic expansion of $H_\Omega(t)$, up to order $4$ in $\sqrt{t}$, as $t\to0$. We remark that we include also the rank-varying case. In order to state our main results, let us introduce the following operator, acting on smooth functions compactly supported close to $\partial\Omega$,
\begin{equation}
N\phi=2g(\nabla\phi,\nabla\delta)+\phi\Delta\delta,
\end{equation} 
where $\delta\colon M\rightarrow\R$ denotes the \sr signed distance function from $\partial\Omega$, see Section \ref{sec:first_ord} for precise definitions. 

\begin{thm}
\label{t:intro1}
Let $M$ be a compact \sr manifold, equipped with a smooth measure $\omega$, and let $\Omega\subset M$ be an open subset whose boundary is smooth and has no characteristic points. Then, as $t\to 0$,
\begin{equation}
\label{eq:intro_exp}
H_\Omega(t) =\omega(\Omega)-\frac{1}{\sqrt\pi}\sigma(\partial\Omega)t^{1/2}
-\frac{1}{12\sqrt\pi}\int_{\partial\Omega}\left(N(\Delta\delta)-2(\Delta\deltas)^2\right) d\area \, t^{3/2}+o(t^2),
\end{equation}
where $\sigma$ denotes the \sr perimeter measure. 
\end{thm}

\begin{rmk}
The compactness assumption in Theorem \ref{t:intro1} is technical and can be relaxed by requiring, instead, global doubling of the measure and a global Poincar\'e inequality, see section \ref{sec:loc_noncomp} and in particular Theorem \ref{thm:doub_poincare}. Some notable examples satisfying these assumptions are:
\begin{itemize}
\item $M$ is a Lie group with polynomial volume growth, the distribution is generated by a family of left-invariant vector fields satisfying the H\"{o}rmander condition and $\omega$ is the Haar measure. This family includes also Carnot groups. 
\item $M=\R^n$, equipped with a \sr structure induced by a family of vector fields $\{Y_1,\ldots,Y_N\}$ with bounded coefficients together with their derivatives, and satisfying the H\"{o}rmander condition. 
\item $M$ is a complete Riemannian manifold, equipped with the Riemannian measure, and with non-negative Ricci curvature.
\end{itemize}
See Section \ref{sec:examples} for further details. In all these examples, Theorem \ref{t:intro1} holds.
\end{rmk}

The strategy of the proof of Theorem \ref{t:intro1} follows a similar strategy of \cite{MR4223354}, inspired by the method introduced in \cite{Savo-heat-cont-asymp}, used for the classical heat content \eqref{eq:def_chc}. However, as we are going to explain in Section \ref{sec:strategy}, new technical difficulties arise, the main one being related to the fact that now $u(t,\cdot)\rvert_{\partial\Omega}\neq0$. At order zero, we obtain the following result, see Section \ref{sec:prel} for precise definitions.
 
\begin{thm}\label{t:intro3}
Let $M$ be a \sr manifold, equipped with a smooth measure $\omega$ and let $\Omega\subset M$ be an open relatively compact subset, whose boundary is smooth and has no characteristic points. Let $x\in\partial\Omega$ and consider a chart of privileged coordinates $\psi\colon U\rightarrow V\subset\R^n$ centered at $x$, such that $\psi(U\cap\Omega)= V\cap\{z_1>0\}$. Then,
\begin{equation}
\lim_{t\to 0}u(t,x)=\int_{\{z_1>0\}}\hat p^x_1(0,z)d\hat\omega^x(z)=\frac{1}{2}, \qquad \forall\,x\in\partial\Omega,
\end{equation}
where $\hat\omega^x$ denotes the nilpotentization of $\omega$ at $x$ and $\hat p_t^x$ denotes the heat kernel associated with the nilpotent approximation of $M$ at $x$ and measure $\hat\omega^x$.
\end{thm}

This result can be seen as a partial generalization of \cite[Prop.\ 3]{MR3022730}, where the authors proved an asymptotic expansion of $u(t,x)$ up to order $1$ in $\sqrt{t}$ for $x\in\partial\Omega$, for a special class of non-characteristic domains in Carnot groups. 

\begin{rmk}
Our proof of Theorem \ref{t:intro3} does not yield an asymptotic series for $u(t,\cdot)\rvert_{\partial\Omega}$ at order higher than $0$. Indeed a complete asymptotic series of this quantity seems difficult to achieve, cf.\ Section \ref{sec:asym_u_bd}. 
\end{rmk}

\begin{rmk}
When $\partial\Omega$ has no characteristic points, the conormal bundle $\mathcal{A}(\partial\Omega):=\{\lambda\in T^*M\mid \langle\lambda,T_{\pi(\lambda)}\partial\Omega\rangle=0\}$ does not intersect the characteristic set and, as a consequence, the principal symbol of the sub-Laplacian is elliptic near $\mathcal{A}(\partial\Omega)$. Thus, it is likely that microlocal analysis techniques in the spirit of \cite{MR3743700} could yield the existence of a complete asymptotic expansion of the relative heat content (but not an explicit expression and geometric interpretation of the coefficients). We thank Yves Colin de Verdi\`ere and the anonymous referee for pointing out this fact.
\end{rmk}

\subsection{Strategy of the proof of Theorem \ref{t:intro1}}
\label{sec:strategy}

To better understand the new technical difficulties in the study of the relative heat content $H_\Omega(t)$, let us compare it with the classical heat content $Q_\Omega(t)$ and illustrate the strategy of the proof of Theorem \ref{t:intro1}.
 
\paragraph{The classical heat content.}
We highlight the differences between the relative heat content $H_\Omega(t)$ and the classical one $Q_\Omega(t)$: let $\Omega\subset M$ an open set in $M$, then for all $t> 0$, we have
\begin{equation}
\label{eq:def_chc}
H_\Omega(t)=\int_\Omega u(t,x)d\omega(x),\qquad Q_\Omega(t)=\int_\Omega u_0(t,x)d\omega(x),
\end{equation} 
where $u(t,x)$ is the solution to \eqref{eqn:cauchy_prob} and $u_0(t,x)$ is the solution to the Dirichlet problem for the heat equation, associated with $\Omega$, i.e.
\begin{equation}\label{eq:dir_prob}
\begin{aligned}
(\partial_t -\Delta)u_0(t,x)  & =  0, & \qquad &\forall (t,x) \in (0,\infty) \times \Omega, \\
u_0(t,x)& = 0,  &  \qquad & \forall (t,x) \in (0,\infty) \times \partial \Omega,\\
u_0(0,x) & 	=  1,  & \qquad & \forall x \in \Omega,
\end{aligned}
\end{equation}
The crucial difference is that $u_0(t,\cdot)\rvert_{\partial\Omega}=0$, for any $t>0$, whereas $u(t,\cdot)\rvert_{\partial\Omega}\neq 0$ in general. Thus, there is no a priori relation between $H_\Omega(t)$ and $Q_\Omega(t)$: the only relevant information is given by domain monotonicity, which implies that:
\begin{equation}
Q_\Omega(t)\leq H_\Omega(t), \qquad\forall\,t>0,
\end{equation}
and clearly this does not give the asymptotics of the latter. See also \cite{MR3116054} for other comparison results in the Euclidean setting. 

\paragraph{Failure of Duhamel's principle.}
In \cite{MR4223354}, we established a complete asymptotic expansion of $Q_\Omega(t)$, as $t\to0$, provided that $\partial\Omega$ has no characteristic points. The proof of this result relied on an iterated application of the Duhamel's principle and the fact that $u_0(t,x)\rvert_{\partial\Omega}=0$. Following the same strategy, we apply Duhamel's principle to a localized version of $H_\Omega(t)$: fix a function $\phi\in C_c^\infty(M)$, compactly supported in a tubular neighborhood around $\partial\Omega$ and such that $0\leq\phi\leq 1$ and $\phi$ is identically $1$, close to $\partial\Omega$. Then, using off-diagonal estimates for the heat kernel, one can prove that:
\begin{equation}
\label{eq:loc_H}
\omega(\Omega)-H_\Omega(t)=I\phi(t,0)+O(t^\infty),\qquad \text{as }t\to 0,
\end{equation}
where $I\phi(t,r)$ is defined for $t>0$ and $r\geq 0$ as
\begin{equation}
\label{eq:Iphi_def}
I\phi(t,r)=\int_{\omNice{r}}(1-u(t,x))\phi(x)d\omega(x),
\end{equation}
here $\omNice{r}=\{x\in \Omega\mid\delta(x)>r\}$, with $\delta\colon\Omega\rightarrow\R$ denoting the distance function from the boundary. Hence, the small-time behavior of $H_\Omega(t)$ is captured by $I\phi(t,0)$. By Duhamel's principle and the \sr mean value lemma, cf.\ Section \ref{sec:first_ord} for details, we obtain the following:
\begin{equation}
\label{eq:1st_exp_I}
I\phi(t,0)=\frac{1}{\sqrt\pi}\int_0^t \int_{\partial\Omega}\left(1-u(\tau,y)\right)\phi(y)d\area(y)(t-\tau)^{-1/2}d\tau+O(t), \qquad\text{as }t\to 0.
\end{equation}
For the classical heat content, $u_0$ satisfies Dirichlet boundary condition, cf.\ \eqref{eq:dir_prob}, hence \eqref{eq:1st_exp_I} would give the first-order asymptotics (and then one could iterate). On the contrary, in this case, we do not have prior knowledge of $u(t,y)$ as $y\in\partial\Omega$ and $t\to0$. Thus, already for the first-order asymptotics, Duhamel's principle alone is not enough, and we need some information on the asymptotic behavior of $u(t,\cdot)\rvert_{\partial\Omega}$.

\paragraph{First-order asymptotics.} We study the asymptotics of $u(t,\cdot)\rvert_{\partial\Omega}$. Using the notion of nilpotent approximation of a \sr manifold, cf. Section \ref{sec:nilpotent}, we deduce the zero-order asymptotic expansion of $u(t,\cdot)\rvert_{\partial\Omega}$ as $t\to 0$, proving Theorem \ref{t:intro3}. This is enough to infer the first-order expansion of $H_\Omega(t)$, by means of \eqref{eq:1st_exp_I}. At this point, we iterate the Duhamel's principle to obtain the higher-order terms of the expansion of $H_\Omega(t)$. However, already at the first iteration, we obtain the following formula for $I\phi$:
\begin{multline}
\label{eq:2nd_exp_I}
I\phi(t,0) =\frac{1}{\sqrt{\pi}}\int_0^t\int_{\partial\Omega}(1-u(\tau,\cdot))\phi \de\area (t-\tau)^{-1/2}d\tau\\
+\frac{1}{2\pi}\int_0^t\int_0^\tau\int_{\partial\Omega}(1-u(\hat\tau,\cdot))N\phi \de\area ((\tau-\hat\tau)(t-\tau))^{-1/2}d\hat\tau \de\tau+O(t^{3/2}),
\end{multline}
as $t\to 0$. Therefore, the zero-order asymptotic expansion of $u(t,\cdot)\rvert_{\partial\Omega}$ no longer suffices for obtaining the second-order asymptotics of $H_\Omega(t)$.

\paragraph{The outside contribution \texorpdfstring{$\op\phi$}{Icphi}.} We mentioned that the crucial difference between $H_\Omega(t)$ and $Q_\Omega(t)$, defined in \eqref{eq:def_chc}, is related to the fact that $u(t,\cdot)\rvert_{\partial\Omega}\neq 0$, whereas $u_0(t,\cdot)\rvert_{\partial\Omega}=0$, for any $t>0$. From a physical viewpoint, this distinction comes from the fact that, since the boundary $\partial\Omega$ is no longer insulated, the heat governed by the Cauchy problem $u(t,x)$, solution to \eqref{eqn:cauchy_prob}, can flow also outside of $\Omega$, whereas $u_0(t,x)$, solution to the Dirichlet problem \eqref{eq:dir_prob}, is confined in $\Omega$, and the external temperature is $0$. Hence, we can imagine that the asymptotic expansion of $H_\Omega(t)$ is affected by the boundary, both from the inside and from the outside of $\Omega$.

Interpreting $I\phi$ as the \emph{inside contribution} to the asymptotics of $H_\Omega$, we are going to formalize the physical intuition of having heat flowing outside of $\Omega$, defining an \emph{outside contribution}, $\op\phi$ to the asymptotics\footnote{The notation ``superscript $c$'' stands for complement. Indeed the outside contribution is the inside contribution of the complement of $\Omega$, see Section \ref{sec:outside}.}. The starting observation is the following simple relation: setting 
\begin{equation}
K_\Omega(t)=\int_{M\setminus\Omega}u(t,x)d\omega(x),\qquad\forall\,t> 0,
\end{equation}
we have, by divergence theorem,
\begin{equation}
\label{eq:key_rel}
H_\Omega(t)+K_\Omega(t)=\omega(\Omega),\qquad\forall\,t> 0.
\end{equation}
Similarly to \eqref{eq:Iphi_def}, for a suitable smooth function $\phi$, one may define a localized version of $K_\Omega(t)$, which we call $\op\phi(t,r)$, so that
\begin{equation}
\label{eq:loc_K}
K_\Omega(t)=\op\phi(t,0)+O(t^\infty),\qquad\text{as }t\to0,
\end{equation}
see Section \ref{sec:outside} for precise definitions. Using \eqref{eq:loc_H}, \eqref{eq:key_rel} and \eqref{eq:loc_K}, we show the following relation:
\begin{equation}
\label{eq:key_rel2}
I\phi(t,0)-\op\phi(t,0)=O(t^\infty),\qquad\text{as }t\to0,
\end{equation}
for a suitable smooth function $\phi$. On the other hand, for the localized quantity $I\phi(t,0)-\op\phi(t,0)$ we have a Duhamel's principle, 
thanks to which we are able to study the asymptotic expansion, up to order $3$, of the \emph{integral} of $u(t,x)$ over $\partial\Omega$, cf.\ Theorem \ref{t:asymp_G}. The limitation to the order $3$ of the asymptotics is technical and seems difficult to overcome, cf. Remark \ref{rmk:limitation}. Inserting this asymptotics in \eqref{eq:2nd_exp_I}, we obtain the asymptotics \emph{up to order $3$} of the expansion of $H_\Omega(t)$, as $t\to 0$. 

\paragraph{Fourth-order asymptotics.} Since we have at disposal only the asymptotics of the integral of $u(t,x)$ over $\partial\Omega$,  up to order $3$, we need a finer argument to obtain the fourth-order asymptotics of $H_\Omega(t)$. The simple but compelling relation is based once again on \eqref{eq:loc_H}, \eqref{eq:key_rel} and \eqref{eq:loc_K}, thanks to which we can write:
\begin{equation} 
\omega(\Omega)-H_\Omega(t)=\frac{1}{2}\left(I\phi(t,0)+\op\phi(t,0)\right)+O(t^\infty),\qquad\text{as }t\to0.
\end{equation}
Now for the sum of the contributions $I\phi(t,0)+\op\phi(t,0)$, the Duhamel's principle implies the following:
\begin{multline}
\label{eq:expr_sum_intro}
I\phi(t,0)+\op\phi(t,0)=\frac{2}{\sqrt{\pi}}\sigma(\partial\Omega) t^{1/2}\\+\frac{1}{2\pi}\int_0^t\int_0^\tau\int_{\partial\Omega}(1-2u(\hat\tau,x))N\phi(y)d\sigma(y)\left((\tau-\hat\tau)(t-\tau)\right)^{-1/2}d\hat\tau \de\tau+o(t).
\end{multline}
This time notice how the integral of $u(t,x)$ over $\partial\Omega$ appears in a first-order term (as opposed to what happened in \eqref{eq:1st_exp_I} or \eqref{eq:2nd_exp_I}), thus its asymptotic expansion up to order $3$ implies a fourth-order expansion for $H_\Omega(t)$, concluding the proof of Theorem \ref{t:intro1}.

\subsection{From the heat kernel asymptotics to the relative heat content asymptotics}
In \cite[Thm.\ A]{YHT-2}, the authors proved the existence of small-time asymptotics of the hypoelliptic heat kernel, $p_t(x,y)$, see Theorem \ref{t:hk_exp} below for the precise statement. In Theorem \ref{t:intro3} we are able to exploit this result to obtain the zero-order asymptotics of the function 
\begin{equation}
u(t,x)=e^{t\Delta}\mathds{1}_\Omega(x)=\int_\Omega p_t(x,y)d\omega(y), \qquad\forall\,t>0,\quad x\in \partial\Omega.
\end{equation}
However, we are not able to extend Theorem \ref{t:intro3} to higher-order asymptotics since, roughly speaking, the remainder terms in Theorem \ref{t:hk_exp} are not uniform as $t\to 0$. If we had a better control on the remainders, we could indeed integrate (in a suitable way) the small-time heat kernel asymptotics to obtain the corresponding expansion for $u(t,x)$. Finally, from such an expansion, the relative heat content asymptotics would follow from the localization principle \eqref{eq:loc_H} and the (iterated) Duhamel's principle \eqref{eq:1st_exp_I}. This is done in Section \ref{sec:asym_u_bd}.

\subsection{Characteristic points}
In order to prove our main results, we need the non-characteristic assumption on the domain $\Omega$. We recall that for a subset $\Omega\subset M$ with smooth boundary, $x\in\partial\Omega$ is a characteristic point if $\D_x\subset T_x(\partial\Omega)$. As was the case for the classical heat content, cf.\ \cite{MR4223354}, the non-characteristic assumption is crucial to follow our strategy, since it guarantees the smoothness of the signed distance function close to $\partial\Omega$, cf. Theorem \ref{thm:sr_tub_neigh}. Nevertheless, one might ask  whether Theorem \ref{t:intro1} holds for domains with characteristic points, at least formally. 

On the one hand, the coefficients, up to order $2$, are well-defined even in presence of characteristic points, cf. \cite{MR2021034}. While, for what concerns the integrand of the third-order coefficient, its integrability, with respect to the \sr induced measure $\sigma$, is related to integrability of the \sr mean curvature $\mathcal{H}$, with respect to the Riemannian induced measure. The latter is a non-trivial property, which has been studied in \cite{DGN-Integrability}, and holds in the Heisenberg group, for surfaces with mildly-degenerate characteristic points in the sense of \cite{integrabilityH}.

On the other hand, differently from what happens in the case of the Dirichlet problem, the heat kernel $p_t(x,y)$ associated with \eqref{eqn:cauchy_prob} is smooth at the boundary of $\Omega$, for positive times, even in presence of characteristic points. Thus, in principle, there is no obstacle in obtaining an asymptotic expansion of $H_\Omega(t)$ also in that case. Moreover, in Carnot groups of step $2$, a similar result to \eqref{eq:de_giorgi} holds, cf.\ \cite{MR2972544,garofalo2020}. In particular, the characterization of sets of finite horizontal perimeter in Carnot groups of step $2$ is independent of the presence of characteristic points, indicating that an asymptotic expansion such as \eqref{eq:intro_exp} may still hold, dropping the non-characteristic assumption. 



\hiddensubsection{Notation}
Throughout the article, for a set $U\subset M$, we will use the notation $C_c^\infty(U)$, even in the compact case, so that all the statements need not be modified in the non-compact case, when the generalization is possible, cf.\ Theorem \ref{thm:doub_poincare}. Moreover, in the non-compact and complete case, the set $\Omega\subset M$ is assumed to be open and bounded.

\paragraph{Acknowledgments.} This work was supported by the Grant ANR-18-CE40-0012 of the ANR, by the Project VINCI 2019 ref.\ c2-1212 and by the European Research Council (ERC) under the European Union's Horizon 2020 research and innovation program (grant agreement No. 945655). \\
We thank Yves Colin de Verdi\`ere, Luc Hillairet and Emmanuel Tr\'elat for stimulating discussion regarding the small-time asymptotics of hypoelliptic heat kernels. We also thank the anonymous referee for valuable comments and remarks.

\section{Preliminaries}
\label{sec:prel}

We recall some essential facts in \sr geometry, following \cite{ABB-srgeom}.

\subsection{Sub-Riemannian geometry}
Let $M$ be a smooth, connected finite-dimensional manifold. A \sr structure on $M$ is defined by a set of $N$ global smooth vector fields $X_1,\ldots,X_N$, called a \emph{generating frame}. The generating frame defines a \emph{distribution} of subspaces of the tangent spaces at each point $x\in M$, given by
\begin{equation}
\label{eqn:gen_frame}
\D_x=\mathrm{span}\{X_1(x),\ldots,X_N(x)\}\subseteq T_xM.
\end{equation}
We assume that the distribution satisfies the \emph{H\"ormander condition}, i.e. the Lie algebra of smooth vector fields generated by $X_1,\dots,X_N$, evaluated at the point $x$, coincides with $T_x M$, for all $x\in M$. The generating frame induces a norm on the distribution at $x$, namely
\begin{equation}
\label{eqn:induced_norm}
\|v\|_g=\inf\left\{\sum_{i=1}^Nu_i^2\mid \sum_{i=1}^Nu_iX_i(x)=v\right\},\qquad\forall\,v\in\D_x,
\end{equation}
which, in turn, defines an inner product on $\D_x$ by polarization, which we denote by $g_x(v,v)$. Let $T>0$. We say that $\gamma \colon [0,T] \to M$ is a \emph{horizontal curve}, if it is absolutely continuous and
\begin{equation}
\dot\gamma(t)\in\D_{\gamma(t)}, \qquad\text{for a.e.}\,t\in [0,T].
\end{equation}
This implies that there exists $u:[0,T]\to\R^N$, such that
\begin{equation}
\dot\gamma(t)=\sum_{i=1}^N u_i(t) X_i(\gamma(t)), \qquad \text{for a.e.}\, t \in [0,T].
\end{equation} 
Moreover, we require that $u\in L^2([0,T],\R^N)$. If $\gamma$ is a horizontal curve, then the map $t\mapsto \|\dot\gamma(t)\|_g$ is integrable on $[0,T]$. We define the \emph{length} of a horizontal curve as follows:
\begin{equation}
\ell(\gamma) = \int_0^T \|\dot\gamma(t)\|_g dt.
\end{equation}
The \emph{\sr distance} is defined, for any $x,y\in M$, by
\begin{equation}\label{eq:infimo}
d_{\mathrm{SR}}(x,y) = \inf\{\ell(\gamma)\mid \gamma \text{ horizontal curve between $x$ and $y$} \}.
\end{equation}
By the Chow-Rashevsky Theorem, the distance $d_{\mathrm{SR}}\colon M\times M\to\R$ is finite and continuous. Furthermore it induces the same topology as the manifold one.
\begin{rmk}
The above definition includes all classical constant-rank sub-Riemannian structures as in \cite{montgomerybook,Riffordbook} (where $\D$ is a vector distribution and $g$ a symmetric and positive tensor on $\D$), but also general rank-varying sub-Riemannian structures. Moreover, the same \sr structure can arise from different generating families. 
\end{rmk}

\subsection{The relative heat content}

Let $M$ be a \sr manifold. Let $\omega$ be a smooth measure on $M$, i.e. by a positive tensor density.  The \emph{divergence} of a smooth vector field is defined by
\begin{equation}
\diverg (X)\omega=\mathcal{L}_X \omega, \qquad\forall\,X\in\Gamma(TM),
\end{equation}
where $\mathcal{L}_X$ denotes the Lie derivative in the direction of $X$. The \emph{horizontal gradient} of a function $f\in C^\infty(M)$, denoted by $\nabla f$, is defined as the horizontal vector field (i.e. tangent to the distribution at each point), such that
\begin{equation}
g_x(\nabla f(x),v)= v(f)(x),\qquad\forall\,v\in\D_x,
\end{equation} 
where $v$ acts as a derivation on $f$. In terms of a generating frame as in \eqref{eqn:gen_frame}, one has
\begin{equation}
\nabla f=\sum_{i=1}^NX_i(f)X_i,\qquad\forall\,f\in C^\infty(M).
\end{equation}
We recall the divergence theorem (we stress that $M$ is not required to be orientable): let $\Omega\subset M$ be open with smooth boundary, then
\begin{equation}
\label{eqn:diverg_thm}
\int_{\Omega}\left(f\diverg X+g(\nabla f,X)\right)d\omega=-\int_{\partial\Omega}fg(X,\nu)d\sigma,
\end{equation}
for any smooth function $f$ and vector field $X$, such that the vector field $fX$ is compactly supported. In \eqref{eqn:diverg_thm}, $\nu$ is the inward-pointing normal vector field to $\Omega$ and $\sigma$ is the induced \sr measure on $\partial\Omega$ (i.e.\ the one whose density is $\sigma=|i_\nu\omega|_{\partial\Omega}$).
  
The \emph{sub-Laplacian} is the operator $\Delta= \diverg\circ\nabla$, acting on $C^\infty(M)$. Again, we may write its expression with respect to a generating frame \eqref{eqn:gen_frame}, obtaining
\begin{equation}
\label{eq:def_delta}
\Delta f=\sum_{i=1}^N\left\{X^2_i(f)+X_i(f)\diverg (X_i)\right\},\qquad\forall\,f\in C^\infty(M).
\end{equation}
We denote by $L^2(M,\omega)$, or simply by $L^2$, the space of real functions on $M$ which are square-integrable with respect to the measure $\omega$. Let $\Omega\subset M$ be an open relatively compact set with smooth boundary. This means that the closure $\bar{\Omega}$ is a compact manifold with smooth boundary. We consider the \emph{Cauchy problem for the heat equation} on $\Omega$, that is we look for functions $u$ such that
\begin{equation}\label{eq:cauchy_prob}
\begin{aligned}
\left(\partial_t -\Delta\right)u(t,x)  & =  0, & \qquad &\forall (t,x) \in (0,\infty) \times M, \\
u(0,\cdot) & 	=  \mathds{1}_\Omega,  & \qquad &\text{in }L^2(M,\omega),
\end{aligned}
\end{equation}
where $u(0,\cdot)$ is a shorthand notation for the $L^2$-limit of $u(t,x)$ as $t\to0$. Notice that $\Delta$ is symmetric with respect to the $L^2$-scalar product and negative, moreover, if $(M,d_\mathrm{SR})$ is complete as a metric space, it is essentially self-adjoint, see \cite{MR862049}. Thus, there exists a unique solution to \eqref{eq:cauchy_prob}, and it can be represented as
\begin{equation}
u(t,x)=e^{t\Delta}\mathds{1}_\Omega(x), \qquad\forall\,x\in M,\ t>0,
\end{equation}
where $e^{t\Delta}\colon L^2\rightarrow L^2$ denotes the heat semi-group, associated with $\Delta$. We remark that for all $\varphi \in L^2$, the function $e^{t\Delta} \varphi$ is smooth for all $(t,x) \in (0,\infty) \times M$, by hypoellipticity of the heat operator and there exists a heat kernel associated with \eqref{eq:cauchy_prob}, i.e. a positive function $p_t(x,y)\in C^\infty((0,+\infty)\times M\times M)$ such that:
\begin{equation}
\label{eq:hk_def}
u(t,x)=\int_M p_t(x,y)\mathds{1}_{\Omega}(y)d\omega(y)=\int_\Omega p_t(x,y)d\omega(y).
\end{equation}

\begin{defn}[Relative heat content]
Let $u(t,x)$ be the solution to \eqref{eq:cauchy_prob}. We define the \emph{relative heat content}, associated with $\Omega$, as
\begin{equation}
\label{eq:rel_heat}
H_\Omega(t)=\int_\Omega u(t,x)d\omega(x), \qquad \forall\,t>0.
\end{equation}
\end{defn}

\begin{rmk}
\label{rmk:closed_set}
If we consider, instead of $\Omega$, a set which is the closure of an open set, then the Cauchy problem \eqref{eq:cauchy_prob} has a unique solution and relative heat content is still well-defined. 
\end{rmk}

We recall here a property of the solution to \eqref{eq:cauchy_prob}: it satisfies a weak maximum principle, meaning that
\begin{equation}
\label{eq:wmax_prin}
0\leq u(t,x)\leq 1, \qquad\forall\,x\in\Omega,\ \forall\,t>0.
\end{equation}
This can be proven following the blueprint of the Riemannian proof (see \cite[Thm.\ 5.11]{MR2569498}). 

\begin{defn}[Characteristic point]
We say that $x\in\partial\Omega$ is a \emph{characteristic point}, or tangency point, if the distribution is tangent to $\partial \Omega$ at $x$, that is
\begin{equation}\label{eqn:char_pts}
\D_x\subseteq T_x(\partial\Omega).
\end{equation}
\end{defn}
We will assume that $\partial\Omega$ has no characteristic points. We say in this case that $\Omega$ is a \emph{non-characteristic domain}. 

\subsection{Nilpotent approximation of \texorpdfstring{$M$}{M}}
\label{sec:nilpotent}

We introduce the notion of nilpotent approximation of a \sr manifold, see \cite{MR3308372,MR1421822} for details. This will be used only in Sections \ref{sec:bd_behav} and \ref{sec:asym_u_bd}.

\paragraph{Sub-Riemannian flag.}
Let $M$ be an $n$-dimensional \sr manifold with distribution $\D$. We define the \emph{flag} of $\D$ as the sequence of subsheafs $\D^k\subset TM$ such that
\begin{equation}
\D^1=\D,\qquad \D^{k+1}=\D^k+[\D,\D^k],\qquad\forall\,k\geq 1,
\end{equation}
with the convention that $\D^0=\{0\}$. Under the H\"ormander condition, the flag of the distribution defines an exhaustion of $T_xM$, for any point $x\in M$, i.e. there exists $r(x)\in\N$ such that:
\begin{equation}
\label{eq:sr_flag}
\{0\}=\D^0_x\subset\D^1_x\subset \ldots\subset \D^{r(x)-1}_x\subsetneq \D^{r(x)}_x=T_xM.
\end{equation}
The number $r(x)$ is called \emph{degree of nonholonomy} at $x$. We set $n_k(x) = \dim \D^k_x$, for any $k\geq 0$, then the collection of $r(x)$ integers 
\begin{equation}
\left(n_1(x),\ldots,n_{r(x)}(x)\right)
\end{equation}
is called \emph{growth vector} at $x$, and we have $n_{r(x)}(x)=n=\dim M$. Associated with the growth vector, we can define the \emph{\sr weights} $w_i(x)$ at $x$, setting for any $i\in\{1,\ldots,n\}$, 
\begin{equation}
\label{eq:sr_weights}
w_i(x)=j,\qquad\text{if and only if}\qquad n_{j-1}(x)+1\leq i\leq n_j(x).
\end{equation}
A point $x\in M$ is said to be \emph{regular} if the growth vector is constant in a neighborhood of $x$, and \emph{singular} otherwise. The \sr structure on $M$ is said to be \emph{equiregular} if all points of $M$ are regular. In this case, the weights are constant as well on $M$. Finally, given any $x\in M$, we define the \emph{homogeneous dimension} of $M$ at $x$ as 
\begin{equation}
\mathcal{Q}(x) = \sum_{i=1}^{r(x)} i ( n_i(x)-n_{i-1}(x) ) =\sum_{i=1}^n w_i(x).
\end{equation}
We recall that, if $x$ is regular, then $\mathcal{Q}(x)$ coincides with the Hausdorff dimension of $(M,d_{\mathrm{SR}})$ at $x$, cf.\ \cite{MR806700}. Moreover, $\mathcal{Q}(x)>n$, for any $x\in M$ such that $\D_x\subsetneq T_xM$. 

\paragraph{Privileged coordinates.}
Let $M$ be a \sr manifold with generating frame \eqref{eqn:gen_frame} and $f$ be the germ of a smooth function $f$ at $x\in M$. We call \emph{nonholonomic derivative} of order $k\in\N$ of $f$, the quantity
\begin{equation}
X_{j_1}\cdots X_{j_k}f(x),
\end{equation}
for any family of indexes $\{j_1,\ldots,j_k\}\subset\{1,\ldots,N\}$. Then, the \emph{nonholonomic order} of $f$ at the point $x$ is 
\begin{equation}
\mathrm{ord}_x(f)=\min\left\{k\in\N\mid \exists\{j_1,\ldots,j_k\}\subset\{1,\ldots,N\}\text{ s.t. }X_{j_1}\cdots X_{j_k}f(x)\neq 0\right\}.
\end{equation}

\begin{defn}[Privileged coordinates]
Let $M$ be a $n$-dimensional \sr manifold and $x\in M$. A system of local coordinates $(z_1,\ldots,z_n)$ centered at $x$ is said to be \emph{privileged} at $x$ if 
\begin{equation}
\mathrm{ord}_x(z_j) = w_j(x),\qquad\forall\, j=1,\ldots,n.
\end{equation}
\end{defn}
Notice that privileged coordinates $(z_1,\ldots,z_n)$ at $x$ satisfy the following property
\begin{equation}
\label{eq:lin_adapted}
{\partial_{z_i}}_{\rvert x} \in \D^{w_i}_x,\qquad {\partial_{z_i}}_{\rvert x}\notin\D^{w_i-1}_x,\qquad\forall i=1,\ldots,n.
\end{equation}
A local frame of $TM$ consisting of $n$ vector fields $\{Z_1,\ldots,Z_n\}$ and satisfying \eqref{eq:lin_adapted} is said to be \emph{adapted} to the flag \eqref{eq:sr_flag} at $x$. Thus, privileged coordinates are always adapted to the flag. In addition, given a local frame adapted to the \sr flag at $x$, say $\{Z_1,\ldots,Z_n\}$, we can define a set of privileged coordinates at $x$, starting from $\{Z_1,\ldots,Z_n\}$, i.e.
\begin{equation}\label{eq:ex_privil}
\R^n\ni(z_1,\ldots,z_n)\mapsto e^{z_1Z_1}\circ\cdots\circ e^{z_nZ_n}(x).
\end{equation}
Moreover, in these coordinates, the vector field $Z_1$ is exactly $\partial_{z_1}$.

\paragraph{Nilpotent approximation.}
Let $M$ be a \sr manifold and let $x\in M$ with weights as in \eqref{eq:sr_weights}. Consider $\psi=(z_1,\ldots,z_n)\colon U\rightarrow V$ a chart of privileged coordinates at $x$, where $U\subset M$ is a relatively compact neighborhood of $x$ and $V\subset\R^n$ is a neighborhood of $0$. Then, for any $\eps\in\R$, we can define the \emph{dilation} at $x$ as 
\begin{equation}
\label{eq:anis_dil}
\delta_\eps\colon\R^n\rightarrow\R^n;\qquad\delta_\eps(z)=\left(\eps^{w_1(x)}z_1,\ldots,\eps^{w_n(x)}z_n\right).
\end{equation} 
Using such dilations, we obtain the nilpotent (or first-order) approximation of the generating frame \eqref{eqn:gen_frame}, indeed setting $Y_i=\psi_*X_i$, for any $i=1\ldots,N$, define
\begin{equation}
\label{eq:nilpot_approx}
\widehat X_i^x =\lim_{\eps\to0}\eps\delta_{\frac{1}{\eps}*}(Y_i),\qquad\forall\,i=1\ldots,N,
\end{equation}
where the limit is taken in the $C^\infty$-topology of $\R^n$. Notice that the vector field $\widehat X_i^x$ is defined on the whole $\R^n$, even though $Y_i$ was defined only on $V\subset\R^n$. 

\begin{thm}
\label{thm:nilp_approx}
Let $M$ be a $n$-dimensional \sr manifold with generating frame $\{X_1,\ldots,X_N\}$ and consider its first-order approximation at $x$ as in \eqref{eq:nilpot_approx}. Then, the frame $\{\widehat X_1^x,\ldots,\widehat X_N^x\}$ of vector fields on $\R^n$ generates a nilpotent Lie algebra of step $r(x)=w_n(x)$ and satisfies the H\"ormander condition.
\end{thm}

The proof of this theorem can be found in \cite{MR3308372}. Recall that a Lie algebra is said to be nilpotent of step $s$ if $s$ is the smallest integer such that all the brackets of length greater than $s$ are zero. 

\begin{defn}[Nilpotent approximation]
Let $M$ be a \sr manifold and let $x\in M$. Then, Theorem \ref{thm:nilp_approx} implies that the frame $\{\widehat X_1^x,\ldots,\widehat X_N^x\}$ is a generating frame for a \sr structure on $\R^n$: we denote the resulting \sr manifold $\widehat M^x$. This is the so-called \emph{nilpotent approximation} of $M$ at the point $x$. 
\end{defn}

Notice that the \sr distance of $\widehat M^x$, denoted by $\hat d^x$, is $1$-homogeneous with respect to the dilations \eqref{eq:anis_dil}.

\begin{rmk}
Up to isometries, the nilpotent approximation of $M$ at $x$ coincides with the Gromov-Hausdorff metric tangent space of $(M,d_{\mathrm{SR}})$ at $x$. Moreover, $\widehat M^x$ is isometric to a quotient of a Carnot group. See \cite{MR1421823,MR1421822,montgomerybook} for further details. 
\end{rmk}

\paragraph{Nilpotentized sub-Laplacian.}
Let $M$ be a \sr manifold, equipped with a smooth measure $\omega$, and let $(z_1,\ldots,z_n)$ be a set of privileged coordinates at $x\in M$. We will use the same symbol $\omega$ to denote measure in coordinates. The \emph{nilpotentization} $\hat{\omega}^x$ of $\omega$ at $x$ is defined as follows:
\begin{equation}
\label{eq:nil_meas}
\langle\hat{\omega}^{x},f\rangle = \lim_{\eps\to 0}\frac{1}{\left|\eps\right|^{\mathcal{Q}(x)}} \langle \delta_\eps^* \omega,f\rangle,\qquad\forall\, f\in C_c^\infty(\R^n).
\end{equation}
Notice that, denoting by $dz=dz_1\ldots dz_n$ the Lebesgue measure on $\R^n$, we have 
\begin{equation}
\label{eq:hom_meas}
\delta_\eps^*(dz)=\left|\eps\right|^{\mathcal{Q}(x)}dz,\qquad\forall\,\eps\neq 0, 
\end{equation}
thus, the limit in \eqref{eq:nil_meas} exists. Finally, we can define the \emph{nilpotentized sub-Lapla\-cian} according to \eqref{eq:def_delta}, acting on $C^\infty(\R^n)$, i.e.
\begin{equation}
\label{eq:delta_nilp}
\widehat{\Delta}^{x} = \mathrm{div}_{\hat \omega^x}\left(\widehat\nabla^x\right)=\sum_{i=1}^N (\widehat{X}^{x}_i)^2.
\end{equation}
We remark that in \eqref{eq:delta_nilp} there is no divergence term, since 
\begin{equation}
\mathrm{div}_{\hat{\omega}^{x}}(\widehat{X}^{x}_i)=0 \qquad \forall i\in\{1,\ldots,N\}.
\end{equation}
As in the general \sr context, in the nilpotent approximation $\widehat M^x$, we may consider the Cauchy heat problem \eqref{eq:cauchy_prob} in $L^2(\R^n,\hat\omega^x)$. We will the denote the associated heat kernel as 
\begin{equation}
\hat p_t^x(z,z')\in C^\infty((0,+\infty)\times\R^n\times\R^n).
\end{equation}

\paragraph{Heat kernel asymptotics.} Let $M$ be a \sr manifold, equipped with a smooth measure $\omega$ and denote by $p_t(x,y)$ the heat kernel \eqref{eq:hk_def}. We have the following result.

\begin{thm}[{\cite[Thm.\ A]{YHT-2}}]
\label{t:hk_exp}
Let $M$ be a \sr manifold and let $\psi\colon U\rightarrow V$ be a chart of privileged coordinates at $x\in M$. Then, for any $m\in\N$,  
\begin{equation}
\label{eq:hk_exp}
|\eps|^{\mathcal{Q}(x)}p_{\eps^2\tau}(\delta_\eps(z),\delta_\eps(z'))=\hat p_\tau^x(z,z')+\sum_{i=1}^m \eps^i f_i^x(\tau,z,z')+o(|\eps|^m),\qquad\text{as }\eps\to 0,
\end{equation}
in the $C^\infty$-topology of $(0,\infty)\times V\times V$, where $f_i^x$'s are smooth functions satisfying the following homogeneity property: for $i=0,\ldots,m$
\begin{equation}
\label{eq:hom_prop}
|\eps|^{\mathcal{Q}(x)}\eps^{-i}f_i^x(\eps^2\tau,\delta_\eps(z),\delta_\eps(z'))=f_i^x(\tau,z,z'), \qquad\forall\,(\tau,z,z')\in (0,\infty)\times \R^n\times \R^n,
\end{equation}
where, for $i=0$, we set $f_0^x(\tau,z,z')=\hat p^x_\tau(z,z')$. In \eqref{eq:hk_exp}, we are considering the heat kernel $p_t$ in coordinates, with a little abuse of notation. 
\end{thm}

\begin{rmk}
We will drop the dependence on the center of the privileged coordinates if there is no confusion.
\end{rmk}

\section{Small-time asymptotics of \texorpdfstring{$u(t,x)$}{u(t,x)} at the boundary}
\label{sec:bd_behav}

We prove here Theorem \ref{t:intro3}, regarding the zero-order asymptotics of $u(t,\cdot)\rvert_{\partial\Omega}$ as $t\to0$.

\begin{thm}\label{t:limit_u}
Let $M$ be a compact \sr manifold, equipped with a smooth measure $\omega$ and let $\Omega\subset M$ be an open subset, whose boundary is smooth and has no characteristic points. Let $x\in\partial\Omega$ and consider a chart of privileged coordinates $\psi\colon U\rightarrow V\subset\R^n$ centered at $x$, such that $\psi(U\cap\Omega)= V\cap\{z_1>0\}$. Then,
\begin{equation}
\label{eq:0th_limit}
\lim_{t\to 0}u(t,x)=\int_{\{z_1>0\}}\hat p^x_1(0,z)d\hat\omega^x(z)=\frac{1}{2}, \qquad \forall\,x\in\partial\Omega,
\end{equation}
where $\hat\omega^x$ denotes the nilpotentization of $\omega$ at $x$ and $\hat p_t^x$ denotes the heat kernel associated with the nilpotent approximation of $M$ at $x$ and measure $\hat\omega^x$.
\end{thm}

\begin{rmk}
A chart of privileged coordinates, such that $\psi(U\cap\Omega)=V\cap\{z_1>0\}$ always exists, provided that $\partial\Omega$ has no characteristic points. Indeed, in this case, there exists a tubular neighborhood of the boundary, cf.\ Theorem \ref{thm:sr_tub_neigh}, which is built through the flow of $\nabla\deltas$, namely 
\begin{equation}
G\colon (-\smallpar,\smallpar)\times\partial\Omega\rightarrow \Omega_{-\smallpar}^{\smallpar};\qquad G(t,q)=e^{t\nabla\deltas}(q),
\end{equation}
is a diffeomorphism such that $G_*\partial_t=\nabla\deltas$ and $\deltas(G(t,q))=t$. Here $\delta\colon M\rightarrow \R$ is the signed distance function\footnote{We warn the reader that $\delta$ without a subscript always denotes the signed distance function and should not be confused with dilations $\delta_\eps$.} from $\partial\Omega$ and $\Omega_{-\smallpar}^{\smallpar}=\{-\smallpar<\delta<\smallpar\}$, see Section \ref{ssec:meanval} for precise definitions. Therefore, choosing an adapted frame for the distribution at $x$, say $\{Z_1,\ldots,Z_n\}$ where $Z_1=\nabla\deltas$, we can define a set of privileged coordinates as in \eqref{eq:ex_privil}:
\begin{equation}
\label{eq:bd_priv}
\R^n\ni(z_1,\ldots,z_n)\mapsto e^{z_1Z_1}\circ \underbrace{e^{z_{2}Z_{2}}\circ\cdots\circ e^{z_nZ_n}(x)}_{\varphi(z_2,\ldots,z_n)}=G(z_1,\varphi(z_2,\ldots,z_n)).
\end{equation} 
The resulting set of coordinates $\psi$ satisfies $\psi_*(\nabla\deltas)=\partial_{z_1}$ and, denoting by $V$ the neighborhood of $0$ in $\R^n$ where $\psi$ is invertible, $\psi(U\cap\Omega)=\{z_1>0\}\cap V$. Here, $e^{sX}(q)$ denotes the flow of the vector field $X$, starting at $q$, evaluated at time $s$.
\end{rmk}

\begin{proof}[Proof of Theorem \ref{t:limit_u}]
Let $p_t(x,y)$ be the heat kernel of $M$, then we may write
\begin{equation}
u(t,x)=\int_\Omega p_t(x,y)d\omega(y), \qquad \forall\,x\in M.
\end{equation}
For a fixed $x\in M$, denoting by $U$ any relatively compact neighborhood of $x$, we have
\begin{equation}
\label{eq:restriction}
\begin{split}
u(t,x)&=\int_{U\cap\Omega} p_t(x,y)d\omega(y)+\int_{\Omega\setminus U}p_t(x,y)d\omega(y)\\
			&=\int_{U\cap\Omega} p_t(x,y)d\omega(y)+O(t^\infty),
\end{split}
\end{equation}
as $t\to 0$. Indeed, since the heat kernel is exponentially decaying outside the diagonal, cf.\ \cite[Prop.\ 3]{JS-estimates}, 
\begin{equation}
\label{eq:js_est}
\int_{\Omega\setminus U}p_t(x,y)d\omega(y)\leq\omega(\Omega\setminus U)C_Ue^{-\frac{c_U}{t}}=O(t^\infty), \qquad\text{as }t\to 0. 
\end{equation}
as $t\to 0$. Now, for $x\in\partial\Omega$, fix the set of privileged coordinates $\psi\colon U\rightarrow V\subset\R^n$, defined as in the statement and assume without loss of generality that $\delta_\eps(V)\subset V$ for any $|\eps|\leq 1$, where $\delta_\eps$ is the dilation \eqref{eq:anis_dil} of the nilpotent approximation of $M$. Also set 
\begin{equation}
V_\eps=\delta_\eps\left(V\cap\{z_1>0\}\right),\qquad\forall\,|\eps|\leq 1.
\end{equation}
when the limits exist, we have:
\begin{equation}
\label{eq:limit_u}
\lim_{t\to 0}u(t,x)=\lim_{t\to 0}\int_{U\cap\Omega} p_t(x,y)d\omega(y)=\lim_{t\to 0}\int_{V_1}  p_t(0,z)d \omega(z),
\end{equation}
where, in the last equation, we are considering the expression of the heat kernel and the measure in coordinates. We want to apply \eqref{eq:hk_exp} at order $1$ in $\eps$, so let us rephrase the statement as follows: for any compact set $K\subset V$,  
\begin{equation}
\label{eq:hk_exp_rem}
|\eps|^{\mathcal{Q}}p_{\eps^2\tau}(0,\delta_\eps(z))=\hat p_\tau(0,z)+\eps R(\eps,\tau,z),\qquad\text{as }\eps\to 0,
\end{equation}
where $R$ is a smooth function such that
\begin{equation}
\label{eq:rem1_est}
\sup_{\substack{\eps\in[-1,1] \\ z\in K}}\left|R(\eps,\tau,z)\right|\leq C(\tau,K),
\end{equation}
with $C(\tau,K)>0$. Notice that \eqref{eq:rem1_est} is not uniform in $\tau$, in the sense that $\tau\mapsto C(\tau,K)$ can explode as $\tau\to0$, in general. Moreover, without loss of generality and, up to restrictions of $U$, we can assume that \eqref{eq:rem1_est} holds globally on $\overline V_1$. For a fixed parameter $L>1$, we set $\tau=1/L$ and $\eps^2=t L$ in \eqref{eq:hk_exp_rem}, obtaining:
\begin{equation}
|t L|^{\mathcal{Q}/2}p_{t}(0,\delta_{\sqrt{t L}}(z))=\hat p_{1/L}(0,z)+\sqrt{t L} R(\sqrt{t L},1/L,z),\qquad\text{as }t\to 0,
\end{equation}
where the remainder $R$ is bounded as $t\to 0$ on the compact sets of $V$, but with a constant depending on $L$. Inserting the above expansion in \eqref{eq:limit_u}, and writing the measure in coordinates $d\omega(z)=\omega(z)dz$ with $\omega(\cdot)\in C^\infty(V_1)$, we have:
\begin{flalign}
u(t,x) & = \int_{V_1}p_t(0,z)\omega(z)dz+O(t^\infty)\\
			 & = \int_{V_{\sqrt{t L}}} p_t(0,z)\omega(z)dz + \int_{V_1\setminus V_{\sqrt{t L}}} p_t(0,z)d\omega(z)  +O(t^\infty) \\
			 & = \int_{V_1} |t L|^{\mathcal{Q}/2}p_{t}(0,\delta_{\sqrt{t L}}(z))\omega(\delta_{\sqrt{t L}}(z))dz + \int_{V_1\setminus V_{\sqrt{t L}}} p_t(0,z)d\omega(z)  +O(t^\infty)\\
			 & = \int_{V_1} \left(\hat p_{1/L}(0,z)+\sqrt{t L} R(\sqrt{t L},1/L,z)\right)\omega(\delta_{\sqrt{t L}}(z))dz\label{eq:SC1}\\
			 & \qquad\qquad\qquad\qquad\qquad\qquad\qquad\qquad\,+ \int_{V_1\setminus V_{\sqrt{t L}}} p_t(0,z)d\omega(z)  +O(t^\infty),\label{eq:SC2}
\end{flalign}
where in the third equality we performed the change of variable $z\mapsto \delta_{1/\sqrt{t L}}(z)$ in the first integral. Let us discuss the terms appearing in \eqref{eq:SC1} and \eqref{eq:SC2}. First of all, for any $L>1$, by definition of the nilpotentization of $\omega$ given in \eqref{eq:nil_meas}, we get
\begin{equation}
\lim_{t\to 0}\int_{V_1}\hat p_{1/L}(0,z)\omega(\delta_{\sqrt{t L}}(z))dz=\int_{V_1}\hat p_{1/L}(0,z)d\hat\omega(z).
\end{equation}
Moreover, for a fixed $L>1$, the integral of $R$ is bounded as $t\to0$, therefore, using \eqref{eq:rem1_est}, we have:
\begin{equation}
\left|\sqrt{t L} \int_{V_1}R(\sqrt{t L},1/L,z)d\omega(z)\right|\leq C_L \sqrt{t},\qquad\forall\,t\leq 1,
\end{equation}
where $C_L>0$ is a constant depending on the fixed $L$. Secondly, by an upper Gaussian bound for the heat kernel in compact \sr manifold, \cite[Thm.\ 2]{JS-estimates}, we obtain the following estimate for \eqref{eq:SC2}:
\begin{equation}
\label{eq:exp_rem}
\int_{V_1\setminus V_{\sqrt{t L}}} p_t(0,z)d\omega(z)\leq \int_{V_1\setminus V_{\sqrt{t L}}}\frac{C_1e^{-\frac{\beta d_{\mathrm{SR}}^2(0,z)}{t}}}{t^{\mathcal{Q}/2}}d\omega(z),
\end{equation}
where $C_1,\beta>0$ are positive constants. Now, by the Ball-Box Theorem \cite[Thm.\ 2.1]{MR3308372}, the \sr distance function at the origin is comparable with the \sr distance of $\widehat M^x$, denoted by $\hat d$. In particular, there exists a constant $c>0$ such that
\begin{equation}
\label{eq:ball-box_thm}
d_{\mathrm{SR}}^2(0,z)\geq c\, \hat d^2(0,z),\qquad\forall\, z\in V.
\end{equation}
Since in \eqref{eq:exp_rem} we are integrating over the set $V_1\setminus V_{\sqrt{t L}}$ and $\hat d$ is 1-homogeneous with respect to $\delta_\eps$, we conclude that
\begin{equation}
d_{\mathrm{SR}}^2(0,z)\geq c\, t L,\qquad\forall\, z\in V_1\setminus V_{\sqrt{t L}}.
\end{equation}
Therefore, using also \eqref{eq:ball-box_thm}, the term \eqref{eq:exp_rem} can be estimated as follows: 
\begin{multline}
\label{eq:gauss_bd}
\int_{V_1\setminus V_{\sqrt{t L}}} p_t(0,z)d\omega(z)\leq C_1 e^{-\frac{c\beta L}{2}}\int_{V_1}\frac{e^{-\frac{\beta d_{\mathrm{SR}}^2(0,z)}{2t}}}{t^{\mathcal{Q}/2}}d\omega(z) \\
 \leq C_1 e^{-\frac{c\beta L}{2}}\int_{V_1}\frac{e^{-\frac{\beta c\, \hat d^2(0,z)}{2t}}}{t^{\mathcal{Q}/2}}\omega(z)dz\leq \tilde C e^{-\frac{c\beta L}{2}},
\end{multline} 
where $\tilde C>0$ is independent of $t$ and $L$. The last inequality in \eqref{eq:gauss_bd} follows from the fact that, after a change of variable $z\mapsto \delta_{1/\sqrt{t}}(z)$, the integral 
\begin{equation}
\int_{V_1}\frac{e^{-\frac{\beta c\, \hat d^2(0,z)}{2t}}}{t^{\mathcal{Q}/2}}\omega(z)dz<+\infty, 
\end{equation}
is uniformly bounded with respect to $t\in [0,\infty)$. 

Therefore, for any $L>1$, we obtain the following estimates for the limit of $u$:
\begin{equation}
\label{eq:ineq_lim}
\begin{split}
\limsup_{t \to 0} u(t,x) & \leq \int_{V_1}\hat p_{1/L}(0,z)d\hat\omega(z) + \tilde C e^{-\frac{c\beta L}{2}}, \\
\liminf_{t \to 0} u(t,x) & \geq \int_{V_1}\hat p_{1/L}(0,z)d\hat\omega(z) - \tilde C e^{-\frac{c\beta L}{2}}.
\end{split}
\end{equation}
In order to evaluate the limits in \eqref{eq:ineq_lim}, let us firstly notice that, since $\hat p$ enjoys upper and lower Gaussian bounds (see for example \cite[App.\ C]{YHT-2}), reasoning as we did for \eqref{eq:gauss_bd}, we can prove the following:
\begin{equation}
\label{eq:intV_1}
\int_{V_1}\hat p_{1/L}(0,z)d\hat\omega(z)=\int_{\{z_1>0\}}\hat p_{1/L}(0,z)d\hat\omega(z)+O\left(e^{-\beta'L}\right).
\end{equation}
Secondly, thanks to \eqref{eq:hom_prop} for $\hat p$, we have the following parity property
\begin{equation}
\hat p_t(0,z)=\hat p_t(0,\delta_{-1}(z)),\qquad\forall\,t>0,\ z\in\R^n,
\end{equation}
and, by the choice of privileged coordinates, $\delta_{-1}(\{z_1>0\})=\{z_1<0\}$. Thus, using also the stochastic completeness of the nilpotent approximation, we obtain for any $t\geq 0$,
\begin{align}
1 &=\int_{\R^n}\hat p_t(0,z)d\hat\omega(z)=\int_{\{z_1>0\}}\hat p_t(0,z)d\hat\omega(z)+\int_{\{z_1<0\}}\hat p_t(0,z)d\hat\omega(z)\\
 &=2\int_{\{z_1>0\}}\hat p_t(0,z)d\hat\omega(z),
\end{align}
having performed the change of variables $z\mapsto\delta_{-1}(z)$ in the last equality. Hence, the integral in \eqref{eq:intV_1} is
\begin{equation}
\int_{V_1}\hat p_{1/L}(0,z)d\hat\omega(z)=\frac{1}{2}+O\left(e^{-\beta'L}\right).
\end{equation}
Finally, we optimize the inequalities \eqref{eq:ineq_lim} with respect to $L$, taking $L\to\infty$ and concluding the proof. 
\end{proof}

\begin{rmk}
In the non-compact case, if $M$ is globally doubling and supports a global Poincar\'e inequality, the proof above is still valid, cf.\ Theorem \ref{thm:doub_poincare}. Otherwise, a different proof is needed, see \cite[App.\ D]{PhDthesis} for details.  
\end{rmk}

\section{First-order asymptotic expansion of \texorpdfstring{$H_\Omega(t)$}{HOmega(t)}}
\label{sec:first_ord}

In this section, we introduce the technical tools that allow us to prove the first-order asymptotic expansion of the relative heat content starting from Theorem \ref{t:limit_u}. The new ingredient is a definition of an operator $I_\Omega$, which depends on the base set $\Omega$.

\subsection{A mean value lemma}
\label{ssec:meanval}
Define $\deltas\colon M\rightarrow \R$ to be the signed distance function from $\partial\Omega$, i.e.
\begin{equation}
\deltas(x)=\begin{cases}
d_\mathrm{SR}(x,\partial\Omega) &x\in \Omega,\\
-d_\mathrm{SR}(x,\partial\Omega) &x\in M\setminus\Omega,  
\end{cases}
\end{equation}
where $d_\mathrm{SR}(\cdot,\partial\Omega)\colon M\rightarrow [0,+\infty)$ denotes the usual distance function from $\partial\Omega$. Let us introduce the following notation: for any $a,b\in\R$, with $a<b$, we set
\begin{equation}
\Omega_a^b=\{x\in M\mid a<\deltas(x)<b\},
\end{equation}
with the understanding that if $b$ (or $a$) is omitted, it is assumed to be $+\infty$ (or $-\infty$), for example\footnote{Notice that the set $\Omega_{-\infty}^{+\infty}=M$, thus omitting both indexes can create confusion. We will never do that and $\Omega$ will always denote the starting subset of $M$.}
\begin{equation}
\Omega_r=\Omega_r^{+\infty}=\{x\in M\mid r<\deltas(x)\}.
\end{equation}
In the non-characteristic case, \cite[Prop.\ 3.1]{FPR-sing-lapl} can be extended without difficulties to the signed distance function.

\begin{thm}[Double-sided tubular neighborhood]
\label{thm:sr_tub_neigh}
Let $M$ be a \sr manifold and let $\Omega\subset M$ be an open relatively compact subset of $M$  whose boundary is smooth and has no characteristic points. Let $\deltas: M \to \R$ be the signed distance function from $\partial \Omega$.
Then, we have:
\begin{itemize}
\item[i)] $\deltas$ is Lipschitz with respect to the \sr distance and $\|\nabla\deltas\|_g\leq 1$ a.e.;
\item[ii)] there exists $\smallpar>0$ such that $\deltas\colon \Omega_{-\smallpar}^{\smallpar}\to\R$ is smooth;
\item[iii)] there exists a smooth diffeomorphism $G: (-\smallpar,\smallpar)\times \partial\Omega \to \Omega_{-\smallpar}^{\smallpar}$, such that
\begin{equation}
\label{tub_neig_main_prop}
\deltas(G(t,y))=t\quad\text{and}\quad G_*\partial_t=\nabla\deltas,\qquad \forall\,(t,y)\in (-\smallpar,\smallpar)\times \partial\Omega.
\end{equation}
Moreover, $\|\nabla \deltas\|_g\equiv 1$ on $\Omega_{-\smallpar}^{\smallpar}$.
\end{itemize}
\end{thm}

In particular, the following co-area formula for the signed distance function holds
\begin{equation}
\label{eq:coarea}
\int_{\omprimeNice{r}}v(x)d\omega(x)=\int_0^r\int_{\bdomNice{s}}v(s,y)d\area(y)ds,\qquad\forall\,r\geq 0,
\end{equation}
where $\area$ is the induced measure on $\bdomNice{s}$, namely the positive measure with density $|i_{\nabla\delta}\omega|_{|\bdomNice{s}}$. From \eqref{eq:coarea}, we deduce the \sr mean value lemma, see \cite[Thm.\ 4.1]{MR4223354} for a proof. 

\begin{prop}
\label{prop:mean_val}
Let $M$ be a compact \sr manifold, equipped with a smooth measure $\omega$, let $\Omega\subset M$ be an open subset of $M$ with smooth boundary and no characteristic points and let $\deltas\colon M\to\R$ be the signed distance function from $\partial\Omega$. Fix a smooth function $v\in C^\infty(M)$ and define
\begin{equation}
\label{eq:meanF}
F(r)=\int_{\omNice{r}}v(x)d\omega(x),\qquad\forall\, r\geq 0.
\end{equation} 
Then there exists $\smallpar>0$ such that the function $F$ is smooth on $[0,\smallpar)$ and, for $0\leq r<\smallpar$:
\begin{equation}
\label{eq:2der_gap}
F''(r)=\int_{\omNice{r}}{\deltasr v(x)d\omega(x)}-\int_{\bdomNice{r}}{v(y)\diverg\left(\nu(y)\right)d\area(y)},
\end{equation}
where $\nu$ is the inward-pointing unit normal to $\omNice{r}$, $\area$ is the induced measure on $\bdomNice{r}$.
\end{prop}


\begin{rmk}
If $v\in C_c^\infty(M)$, then neither $M$ nor $\Omega$ is required to be compact for Proposition \ref{prop:mean_val} to be true, indeed its proof relies on \eqref{eq:coarea}, which continues to hold, and the divergence theorem \eqref{eqn:diverg_thm}, which applies if $\mathrm{supp}(v)$ is compact. Moreover, we remark that $\nu_r$ is equal to $\nabla\delta$ up to sign. We prefer to keep $\nu_r$ in \eqref{eq:meanF}, since we are going to apply it when the integral is performed over $\omNice{r}$ or its complement.  
\end{rmk}

If we choose the function $v$ in the definition of $F$ to be $1-u(t,x)$, where $u(t,\cdot)=e^{t\Delta}\mathds{1}_{\Omega}$, then, $F$ satisfies a non-homogeneous one-dimensional heat equation.

\begin{cor}
\label{cor:nonhom_heat_old}
Under the hypotheses of Proposition \ref{prop:mean_val}, the function
\begin{equation}
\label{eq:defn_F}
F(t,r)=\int_{\omNice{r}}\left(1-u(t,x)\right)d\omega(x),\qquad\forall\,t>0,\quad r\geq 0,
\end{equation}
where $u(t,x)=e^{t\Delta}\mathds{1}_{\Omega}(x)$, satisfies the following non-homogeneous one-dimensional heat equation:
\begin{equation}
\label{eq:nonhom_heat}
(\partial_t-\partial_r^2)F(t,r)=\int_{\bdomNice{r}}{\left(1-u(t,\cdot)\right)\diverg(\nu) d\area}, \qquad t > 0,\quad r\in[0,\smallpar).
\end{equation}
Here $\nu$ is the inward-pointing unit normal to $\omNice{r}$, $\area$ is the induced measure on $\bdomNice{r}$.
\end{cor}

Corollary \ref{cor:nonhom_heat_old} holds only for $r\leq\smallpar$, however we would like to extend it to the \emph{whole} positive half-line, in order to apply a Duhamel's principle. This can be done up to an error which is exponentially small.

\subsection{Localization principle}

\begin{prop}
\label{prop:easy_loc}
Let $M$ be a compact \sr manifold, equipped with a smooth measure $\omega$, and let $\Omega\subset M$ be an open subset of $M$, with smooth boundary. Moreover, let $K\subset M$ be a closed set such that $K\cap\partial\Omega=\emptyset$. Then
\begin{equation}
\mathds{1}_\Omega(x)-u(t,x)=O(t^\infty), \qquad\text{uniformly for }x\in K,
\end{equation}
where $u(t,x)=e^{t\Delta}\mathds{1}_{\Omega}(x)$.
\end{prop}

\begin{proof}
The statement is a direct consequence of the off-diagonal estimate for the heat kernel in compact \sr manifold (see \cite[Prop.\ 3]{JS-estimates}):
\begin{equation}\label{prop3Jerison}
p_t(x,y) \leq C_a e^{-c_a/t}, \qquad \forall\,x,y \text{ with } d(x,y) \geq a,\quad t< 1,
\end{equation}
for suitable constants $C_a,c_a>0$, depending only on $a$. Now, since $K\cap\partial\Omega=\emptyset$, we can write $K$ as a disjoint union
\begin{equation}
K=K_1\sqcup K_2\qquad\text{with }K_1\subset\Omega,\quad K_2\subset M\setminus\Omega.
\end{equation}
At this point, for $i=1,2$, set $a_i= d_\mathrm{SR}(K_i,\partial\Omega)>0$ by hypothesis, and let $x\in K_1$. Then, using the stochastic completeness of $M$, we have:
\begin{equation}
\label{eq:loc_estimate}
\left|\mathds{1}_\Omega(x)-u(t,x)\right|=1-u(t,x)=\int_{M\setminus\Omega}p_t(x,y)d\omega(y)\leq C_1 e^{-c_1/t}\omega(M\setminus\Omega),
\end{equation}
which is exponentially decaying, uniformly in $K_1$. Analogously, if $x\in K_2$, we have
\begin{equation}
\left|\mathds{1}_\Omega(x)-u(t,x)\right| = u(t,x)=\int_Mp_t(x,y)\mathds{1}_\Omega(y) d\omega(y)=\int_\Omega p_t(x,y)d\omega(y)\leq C_2 e^{-c_2/t}\omega(\Omega),
\end{equation} 
uniformly in $K_2$.
\end{proof}

\begin{rmk}
In the non-compact case, Proposition \ref{prop:easy_loc} may fail. Indeed, on the one hand the off-diagonal estimate \eqref{prop3Jerison} is not always available, on the other hand the measure of $M\setminus\Omega$ appearing in \eqref{eq:loc_estimate} is infinite. Under additional assumption on $M$, we are able to recover a localization principle, see Section \ref{sec:loc_noncomp}.
\end{rmk}

Let $M$ be compact. Thanks to Proposition \ref{prop:easy_loc}, we can extend the function $F$ defined in \eqref{eq:defn_F}, to a solution to a non-homogeneous heat equation such as \eqref{eq:nonhom_heat} on the whole half-line. More precisely, let $\phi,\eta\in C_c^\infty(M)$ such that
\begin{equation}
\label{eq:cutoff}
\phi + \eta  \equiv 1, \qquad \mathrm{supp}(\phi) \subset \Omega_{-\smallpar}^{\smallpar}, \qquad \mathrm{supp}(\eta) \subset \omext{\smallpar/2}\cup\omNice{\smallpar/2},
\end{equation}
where $\smallpar$ is defined in Proposition \ref{prop:mean_val}. We have then, for $r\in[0,\smallpar)$,
\begin{align}
F(t,r) & =\int_{\omNice{r}}\left(1- u(t,x)\right)\phi(x) d\omega(x) +  \int_{\omNice{r}}\left(1- u(t,x)\right)\eta(x) d\omega(x) \\
& = \int_{\omNice{r}}\left(1- u(t,x)\right)\phi(x) d\omega(x) +  \int_{\mathrm{supp}(\eta)\cap\omNice{r}} \left(1- u(t,x)\right)\eta(x) d\omega(x) \\
& = \int_{\omNice{r}} \left(1- u(t,x)\right)\phi(x) d\omega(x) + O(t^\infty), \label{eq:split}
\end{align}
where we used Proposition \ref{prop:easy_loc} to deal with the second term, having set $K=\mathrm{supp}(\eta)\cap\omNice{r}$. For this reason, we may focus on the first term in \eqref{eq:split}. 

\begin{defn}
\label{def:ILambda_om}
For all $t> 0$ and $r\geq 0$, we define the operators $I_\Omega,\Lambda_\Omega:\extfunsp\to C^\infty((0,\infty)\times [0,\infty))$, associated with $\Omega$, by
\begin{align}
I_\Omega\phi(t,r) &= \int_{\omNice{r}}{\left(1- u(t,x)\right)\phi(x)d\omega(x)},\\
\Lambda_\Omega\phi(t,r) &=-\partial_rI_\Omega\phi(t,r)=-\int_{\bdomNice{r}}{\left(1- u(t,y)\right)\phi(y)d\sigma(y)},
\end{align}
for any $\phi\in \extfunsp$, and where $\sigma$ denotes the induced measure on $\bdomNice{r}$ and $u(t,\cdot)=e^{t\Delta}\mathds{1}_\Omega(\cdot)$.
\end{defn}

\begin{rmk}
We stress that, for every $\phi\in\extfunsp$, $I_\Omega\phi$, $\Lambda_\Omega\phi$ are indeed smooth in both variables thanks to the choice of the parameter $\smallpar>0$ as in Proposition \ref{prop:mean_val}, together with the smoothness of the solution to the heat equation. Moreover, $\Lambda_\Omega\phi$ is compactly supported in the $r$-variable. 
\end{rmk}

Thanks to the localization principle, we can improve Corollary \ref{cor:nonhom_heat_old}, obtaining a better result for $I_\Omega\phi(t,r)$. 
\begin{lem}
\label{lem:LI_om}
Let $L= \partial_t-\partial_{r}^2$ be the one-dimensional heat operator. Then, for any $\phi\in\extfunsp$, 
\begin{equation}
\label{eq:LI_Omega}
L\left(I_\Omega\phi(t,r)\right)=I_\Omega\Delta\phi(t,r)+\Lambda_\Omega N_\Omega\phi(t,r),\qquad\forall \,t>0,\quad r\geq 0,
\end{equation} 
where $N_\Omega$ is the operator defined by:
\begin{equation}
\label{eq:sr_N}
N_\Omega\phi= 2g\left(\nabla\phi,\nu\right)+\phi\,\diverg(\nu),\qquad\forall\,\phi\in \extfunsp,
\end{equation}
and $\nu$ is the inward-pointing unit normal to $\Omega$.
\end{lem}

\subsection{Duhamel's principle for \texorpdfstring{$I_\Omega\phi$}{IOmega phi}}

We recall for the convenience of the reader a one-dimensional version of the Duhamel's principle, see \cite[Lem.\ 5.4]{MR4223354}.

\begin{lem}[Duhamel's principle]
\label{lem:neg_duham}
Let $f\in C((0,\infty)\times [0,\infty))$, $v_0,v_1\in C([0,\infty))$, such that $f(t,\cdot)$ and $v_0$ are compactly supported and assume that
\begin{equation}
\label{eq:lim_source}
\exists\lim_{t\to 0} f(t,r),\qquad\forall r\geq 0.	
\end{equation}
Consider the non-homogeneous heat equation on the half-line:
\begin{equation}
\begin{aligned}\label{eq:neum_prob}
Lv(t,r) & = f(t,r), & \qquad & \text{for }t>0,\ r>0, \\
v(0,r)&= v_0(r),&  \qquad  & \text{for }r>0,\\
\partial_rv(t,0)&= v_1(t),& \qquad  & \text{for }t>0,
\end{aligned}
\end{equation}
where $L=\partial_t-\partial_{r}^2$. Then, for $t>0$ and $r\geq 0$, the solution to \eqref{eq:neum_prob} is given by
\begin{multline}
\label{eq:duham_prin}
v(t,r)=\int_0^\infty e(t,r,s)v_0(s)ds +\int_0^t\int_0^\infty e(t-\tau,r,s)f(\tau,s)ds \de\tau \\
-\int_0^te(t-\tau,r,0)v_1(\tau)d\tau,
\end{multline}
where $e(t,r,s)$ is the Neumann heat kernel on the half-line, that is
\begin{equation}\label{eq:neumanheat}
e(t,r,s) = \frac{1}{\sqrt{4\pi t}}\left(e^{-(r-s)^2/4t}+e^{-(r+s)^2/4t}\right).
\end{equation}
\end{lem}

Finally, we apply Lemma \ref{lem:neg_duham} to obtain an asymptotic equality for $I_\Omega\phi$. The main difference with the result of \cite[Thm.\ 5.6]{MR4223354} is that the former will not be a true first-order asymptotic expansion.  

\begin{cor}
\label{cor:1st_exp}
Let $M$ be a compact \sr manifold, equipped with a smooth measure $\omega$, and let $\Omega\subset M$ be an open subset whose boundary is smooth and has no characteristic points. Then, for any function $\phi\in\extfunsp$, 
\begin{equation}
\label{eq:1st_exp}
I_\Omega\phi(t,0)=\frac{1}{\sqrt\pi}\int_0^t \int_{\partial\Omega}\left(1-u(\tau,y)\right)\phi(y)d\area(y)(t-\tau)^{-1/2}d\tau+O(t),
\end{equation}
as $t\to 0$, where $u(t,\cdot)=e^{t\Delta}\mathds{1}_\Omega(\cdot)$.
\end{cor}

\begin{proof}
By Lemma \ref{lem:LI_om}, the function $I_\Omega\phi(t,r)$ satisfies the following Neumann problem on the half-line:
\begin{equation}
\begin{aligned}
LI_\Omega\phi(t,r) & = f(t,r), & \qquad & \text{for }t>0,\ r>0, \\
I_\Omega\phi(0,r)&= 0,&  \qquad  & \text{for }r>0,\\
\partial_rI_\Omega\phi(t,0)&= -\Lambda_\Omega\phi(t,0),& \qquad  & \text{for }t>0,
\end{aligned}
\end{equation}
where the source is given by $f(t,r)=I_\Omega\deltasr\phi(t,r)+\Lambda_\Omega N_\Omega\phi(t,r)$. Thus, applying Duhamel's formula \eqref{eq:duham_prin}, we have:
\begin{equation}
I_\Omega\phi(t,0)=\int_0^t\int_0^{+\infty}e(t-\tau,0,s)f(\tau,s)ds \de\tau+\frac{1}{\sqrt{\pi}}\int_0^t \frac{1}{\sqrt{t-\tau}}\Lambda_\Omega\phi(t,0)d\tau.
\end{equation}
Since the source is uniformly bounded by the weak maximum principle \eqref{eq:wmax_prin}, the first integral is a remainder of order $t$, as $t\to 0$, concluding the proof. 
\end{proof}

\begin{rmk}
We mention that a relevant role in the sequel will be played by the operators $I_\Omega$, cf.\ Definition \ref{def:ILambda_om}, associated with either $\Omega$ or its complement $\Omega^c$. 
\end{rmk}

\subsection{First-order asymptotics}

In this section we prove the first-order asymptotic expansion of $H_\Omega(t)$, cf.\ Theorem \ref{t:intro1} at order $1$. We will use Corollary \ref{cor:1st_exp}, for the \emph{inside contribution}:
\begin{equation}
\label{eq:inside}
I\phi(t,r) = \int_{\omNice{r}}{\left(1- u(t,x)\right)\phi(x)d\omega(x)},\qquad\forall\,t>0,\quad r\geq 0,
\end{equation}
for any $\phi\in \extfunsp$, and where $\sigma$ denotes the induced measure on $\bdomNice{r}$ and $u(t,\cdot)=e^{t\Delta}\mathds{1}_\Omega(\cdot)$ is the solution to \eqref{eq:cauchy_prob}. The quantity \eqref{eq:inside} is just Definition \ref{def:ILambda_om}, applied with base set $\Omega\subset M$.

\begin{thm}
\label{thm:1st_exp_H}
Let $M$ be a compact \sr manifold, equipped with a smooth measure $\omega$, and let $\Omega\subset M$ be an open subset whose boundary is smooth and has no characteristic points. Then,
\begin{equation}
\label{eq:1st_exp_H}
H_\Omega(t)=\omega(\Omega)-\frac{1}{\sqrt\pi}\area(\partial\Omega) t^{1/2}+o(t^{1/2}), \qquad\text{as }t\to 0.
\end{equation}
\end{thm}


\begin{proof}
Let $\phi\in\extfunsp$ be as in \eqref{eq:cutoff}, namely 
\begin{equation}
0\leq\phi\leq 1,\qquad\text{and}\qquad\phi\equiv 1\quad\text{in}\quad\Omega_{-\smallpar/2}^{\smallpar/2}.
\end{equation}
Then, by the localization principle, cf.\ \eqref{eq:split}, we have that 
\begin{equation}
\label{eq:true_loc_H}
\omega(\Omega)-H_\Omega(t)=I\phi(t,0)+O(t^\infty),\qquad\text{as }t\to0.
\end{equation}
Applying Corollary \ref{cor:1st_exp}, we have:
\begin{equation}
\label{eq:expr_I1}
I\phi(t,0)=\frac{1}{\sqrt\pi}\int_0^t \int_{\partial\Omega}\left(1-u(\tau,y)\right)\phi(y)d\area(y)(t-\tau)^{-1/2}d\tau+O(t), \qquad\text{as }t\to 0.
\end{equation}
Thus, to infer the first-order term of the asymptotic expansion we have to compute the following limit:
\begin{equation}
\label{eq:lim_1coeff}
\lim_{t\to 0}\frac{I\phi(t,0)}{t^{1/2}}=\lim_{t\to0}\frac{1}{t^{1/2}\sqrt\pi}\int_0^t \int_{\partial\Omega}\left(1-u(\tau,y)\right)\phi(y)d\area(y)(t-\tau)^{-1/2}d\tau.
\end{equation}
Firstly, by the change of variable in the integral $\tau\mapsto t\tau$, we rewrite the argument of the limit \eqref{eq:lim_1coeff} as
\begin{equation}
\frac{1}{\sqrt{\pi} }\int_0^1 \int_{\partial\Omega}\left(1-u(t\tau,y)\right)\phi(y)d\area(y)(1-\tau)^{-1/2}d\tau.
\end{equation}
Secondly, we apply the dominated convergence theorem. Indeed, on the one hand, by Theorem \ref{t:limit_u} we have point-wise convergence
\begin{equation}
\left(1-u(t\tau,y)\right)\phi(y)\xrightarrow{t\to 0}\frac{1}{2}\phi(y), \qquad\forall\, y\in\partial\Omega, \quad\tau\in (0,1),
\end{equation}
and on the other hand, by the maximum principle
\begin{equation}
\left|\int_{\partial\Omega}\left(1-u(t\tau,y)\right)\phi(y)d\area(y)(1-\tau)^{-1/2}\right|\leq \int_{\partial\Omega}|\phi| d\area (1-\tau)^{-1/2}\in L^1(0,1),
\end{equation}
for any $t>0$. Therefore, we finally obtain that:
\begin{equation}
\label{eq:aux_1st_ord}
I\phi(t,0)=\sqrt{\frac{t}{\pi}} \int_{\partial\Omega}\phi(y)d\area(y)+o(t^{1/2}), \qquad\text{as }t\to 0.
\end{equation}
Recalling that $\phi_{|\partial\Omega}\equiv 1$, we conclude the proof.
\end{proof}

\begin{rmk}
The above technique used to evaluate the first-order coefficient causes a loss of precision in the remainder, with respect to the expression \eqref{eq:expr_I1}, where the remainder is $O(t)$. This loss comes from the application of Theorem \ref{t:limit_u}, which does not contain any remainder estimate. 
\end{rmk}

\section{Higher-order asymptotic expansion of \texorpdfstring{$H_\Omega(t)$}{HOmega(t)}}

We iterate Duhamel's formula \eqref{eq:duham_prin} for the inside contribution to study the higher-order asymptotics of $H_\Omega(t)$. We obtain the following expression for $I\phi$, at order $3$: 
\begin{multline}
\label{eq:expr_I}
I\phi(t,0) =\frac{1}{\sqrt{\pi}}\int_0^t\int_{\partial\Omega}(1-u(\tau,\cdot))\phi \de\area (t-\tau)^{-1/2}d\tau\\
					 +\frac{1}{2\pi}\int_0^t\int_0^\tau\int_{\partial\Omega}(1-u(\hat\tau,\cdot))N\phi \de\area ((\tau-\hat\tau)(t-\tau))^{-1/2}d\hat\tau \de\tau+O(t^{3/2}),
\end{multline}
where $u(t,\cdot)=e^{t\Delta}\mathds{1}_\Omega(\cdot)$ denotes the solution to \eqref{eq:cauchy_prob} and $N$ is the operator defined by
\begin{equation}
\label{eq:sr_N_true}
N\phi=2g(\nabla\phi,\nabla\deltas)+\phi\Delta\deltas,\qquad\forall\,\phi\in \extfunsp,
\end{equation}
with $\delta\colon M\rightarrow \R$ the signed distance function from $\partial\Omega$. The computations for deriving \eqref{eq:expr_I} are technical. We refer to Appendix \ref{app:computations} for further details, and in particular to Lemma \ref{lem:aux_exp2}. Motivated by \eqref{eq:expr_I}, we introduce the following functional.

\begin{defn}
Let $M$ be a \sr manifold, equipped with a smooth measure $\omega$, let $\Omega\subset M$ be a relatively compact subset with smooth boundary and let $v\in C^\infty((0,+\infty)\times M)$. Define the functional $\G_v$, for any $\phi\in\extfunsp$ as:
\begin{equation}
\label{eq:def_G}
\G_v[\phi](t)=\frac{1}{2\sqrt{\pi}}\int_0^t\int_{\partial\Omega}v(\tau,\cdot)\phi \de\area (t-\tau)^{-1/2}d\tau,\qquad\forall\,t\geq 0,
\end{equation}
where $\sigma$ is the \sr induced measure on $\partial\Omega$.
\end{defn}

Notice that the functional $\G_v$ is linear with respect to the subscript function $v$, by linearity of the integral. Moreover, when the function $v$ is chosen to be the solution to \eqref{eq:cauchy_prob}, we easily obtain the following corollary of Theorem \ref{t:limit_u}, which is just a rewording of \eqref{eq:lim_1coeff}.

\begin{cor}
\label{lem:limit_G}
Let $M$ be a compact \sr manifold, equipped with a smooth measure $\omega$, and let $\Omega\subset M$ be an open subset whose boundary is smooth and has no characteristic points. Let $\phi\in \extfunsp$, then, 
\begin{equation}
\label{eq:limit_G1}
\G_{u}[\phi](t)=\frac{1}{2\sqrt{\pi}}\int_{\partial\Omega}\phi(y)d\sigma(y) t^{1/2}+o(t^{1/2}),\qquad\text{as }t\to 0.
\end{equation}
\end{cor}

Then, we can rewrite \eqref{eq:expr_I} in a compact notation:
\begin{equation}
\label{eq:expr_I_G}
I\phi(t,0) =2\G_{1-u}[\phi](t)+\frac{1}{\sqrt\pi}\int_0^t\G_{1-u}[N\phi](t)d\area (t-\tau)^{-1/2} d\tau+O(t^{3/2}).
\end{equation}
However, on the one hand, the application of Corollary \ref{lem:limit_G} to \eqref{eq:expr_I} does not give any new information on the asymptotics of $H_\Omega(t)$, as the first term produces an error of $o(t^{1/2})$. On the other hand, it is clear that an asymptotic series of $\G_u$ is enough to deduce the small-time expansion of $H_\Omega(t)$. 

\subsection{The outside contribution and an asymptotic series for \texorpdfstring{$\G_u[\phi]$}{Gu[phi]}}
\label{sec:outside}

In this section, we deduce an asymptotic series, at order $3$, of $\G_u[\phi](t)$ as $t\to 0$. This is done exploiting the fact that the diffusion of heat is not confined in $\Omega$, and as a result we can define an \emph{outside contribution}, namely the quantity obtained from Definition \ref{def:ILambda_om}, applied with base set $\Omega^c\subset M$:
\begin{equation}
\label{eq:outside}
\op\phi(t,r) = \int_{(\Omega^c)_r}{\left(1- u^c(t,x)\right)\phi(x)d\omega(x)},\qquad\forall\,t>0,\quad r\geq 0,
\end{equation}
for any $\phi\in \extfunsp$, and where $\sigma$ denotes the induced measure on the boundary of $(\Omega^c)_r$ and $u^c(t,x)=e^{t\Delta}\mathds{1}_{\Omega^c}(x)$. We remark that, since $\Omega$ and its complement share the boundary, then $(\Omega^c)_{-\smallpar}^{\smallpar}=\Omega_{-\smallpar}^{\smallpar}$. 
It is convenient to introduce \eqref{eq:outside}, because it turns out that the quantity $I\phi-\op\phi$, where $I\phi$ is the inside contribution \eqref{eq:inside}, has an explicit asymptotic series in integer powers of $t$.

\begin{prop}
\label{prop:asym_rel_tmp}
Let $M$ be a compact \sr manifold, equipped with a smooth measure $\omega$, and let $\Omega\subset M$ be an open subset with smooth boundary. Let $\phi\in \extfunsp$, then, for any $k\in\N$,
\begin{equation}
\label{eq:asym_rel_tmp}
I\phi(t,0)-\op\phi(t,0)=\sum_{i=1}^ka_i(\phi)t^i+O(t^{k+1}),\qquad\text{as }t\to0,
\end{equation}
where 
\begin{equation}
\begin{aligned}
\label{eq:coeff_tmp}
a_i(\phi) &= \int_{\partial\Omega}g(\nabla(\Delta^{i-1}\phi),\nabla\deltas)d\area, \qquad\text{for }i\geq 1.
\end{aligned}
\end{equation}
\end{prop}

\begin{proof}
Recall that in the definition of the outside contribution \eqref{eq:outside}, the integrand function involves $u^c(t,x)=e^{t\Delta}\mathds{1}_{\Omega^c}(x)$. Since $M$ is compact, and hence stochastically complete, we have: 
\begin{equation}
\label{eq:easy_rel_u}
1-u^c(t,x)=e^{t\Delta}1(x)-e^{t\Delta}\mathds{1}_{\Omega^c}(x)=u(t,x),\qquad\forall t>0,\quad x\in M,
\end{equation}
having used the point-wise equality $1-\mathds{1}_{\Omega^c}=\mathds{1}_\Omega$ in $M\setminus\partial\Omega$. Therefore, we can write the difference $I\phi(t,0)-\op\phi(t,0)$ as follows:
\begin{align}
I\phi(t,0)-\op\phi(t,0) &=\int_{\Omega}\left(1- u(t,\cdot)\right)\phi \de\omega-\int_{\Omega^c}\left(1- u^c(t,\cdot)\right)\phi \de\omega\\
												&=\int_{\Omega}\left(1- u(t,\cdot)\right)\phi \de\omega-\int_{\Omega^c}u(t,\cdot)\phi \de\omega\\
												&=\int_\Omega\phi(x)d\omega(x)-\int_Mu(t,x)\phi(x)d\omega(x).\label{eq:aux_fun_tmp}
\end{align}
Since $u(t,x)$ is the solution to \eqref{eq:cauchy_prob}, the function \eqref{eq:aux_fun_tmp} is smooth as $t\in[0,\infty)$. Indeed, the smoothness in the open interval is guaranteed by hypoellipticity of the sub-Laplacian. At $t=0$, the divergence theorem, together with the fact that $\phi$ has compact support in $M$, implies that
\begin{equation}
\label{eq:taylor_sum_tmp}
\begin{split}
\partial_t^i\left(\int_Mu(t,x)\phi(x)d\omega(x)\right) &=\int_M\partial_t^i\left(u(t,x)\phi(x)\right)d\omega(x)=\int_M\Delta^iu(t,x)\phi(x)d\omega(x)\\
																																					&=\int_Mu(t,x)\Delta^i\phi(x)d\omega(x)\xrightarrow{t\to0}\int_\Omega\Delta^i\phi(x)d\omega(x).
\end{split}
\end{equation}
The previous limit shows that \eqref{eq:aux_fun_tmp} is smooth at $t=0$, and also that its asymptotic expansion at order $k$, as $t\to0$, coincides with its $k$-th Taylor polynomial at $t=0$. Finally, we recover \eqref{eq:asym_rel_tmp}, applying once again the divergence theorem:
\begin{equation}
\int_\Omega\Delta^i\phi \de\omega=-\int_{\partial\Omega}g(\nabla(\Delta^{i-1}\phi),\nu)d\area=-\int_{\partial\Omega}g(\nabla(\Delta^{i-1}\phi),\nabla\deltas)d\area,
\end{equation} 
recalling that $\nu = \nabla\delta$ is the inward-pointing normal vector to $\Omega$ at its boundary.
\end{proof}

Applying the (iterated) Duhamel's principle \eqref{eq:duham_prin} to the difference $I\phi-\op\phi$, we are able to obtain relevant information on the functional $\G_u$.
 
\begin{thm}
\label{t:asymp_G}
Let $M$ be a compact \sr manifold, equipped with a smooth measure $\omega$, and let $\Omega\subset M$ be an open subset whose boundary is smooth and has no characteristic points. Then, for any $\phi\in \extfunsp$, 
\begin{equation}
\label{eq:asymp_G}
\G_{u}[\phi](t)=\frac{1}{2\sqrt\pi}\int_{\partial\Omega}\phi \de\area \, t^{1/2}+\frac{1}{8}\int_{\partial\Omega}\phi\Delta\deltas \de\area \, t+o(t^{3/2}),\qquad\text{as }t\to 0.
\end{equation}
\end{thm}

\begin{proof}
Let us study the difference of the inside and outside contributions $I\phi(t,0)-\op\phi(t,0)$: on the one hand, we have an iterated Duhamel's principle, cf.\ Lemma \ref{lem:formal_exp_diff}, which we report here:
\begin{flalign}
\left(I\phi-\op\phi\right)(t,0) &= 2\G_{1-2u}[\phi](t)+\frac{1}{2}\int_{\partial\Omega}N\phi \de\area \, t\label{eq:diff_exp_aux}\\
									&\quad+\frac{1}{2\pi}\int_0^t \int_0^\tau \G_{1-2u}[N^2\phi](\hat\tau)\left((\tau-\hat\tau)(t-\tau)\right)^{-1/2}d\hat\tau \de\tau\\
									&\quad+\frac{1}{4\sqrt\pi}\int_0^t\int_{\partial\Omega}\left(1-2u(\tau,\cdot)\right)(4\deltasr-N^2)\phi \de\area(t-\tau)^{1/2}d\tau+O(t^2),
\end{flalign}
where we recall that $N$ is the operator acting on smooth functions compactly supported close to $\partial\Omega$ defined by
\begin{equation}
N\phi=2g(\nabla\phi,\nabla\deltas)+\phi\Delta\deltas,\qquad\forall\,\phi\in \extfunsp.
\end{equation}
Using Corollary \ref{lem:limit_G} and the linearity of $\G_v$ with respect to $v$, we know that 
\begin{equation} 
\label{eq:asymp_aux1}
\G_{1-2u}[\phi](t)=o(t^{1/2}), \qquad\text{as }t\to 0,\quad\forall\,\phi\in\extfunsp.
\end{equation}
Therefore, applying \eqref{eq:asymp_aux1} to the function $N^2\phi\in\extfunsp$, we obtain 
\begin{equation}
\label{eq:asymp1}
\frac{1}{2\pi}\int_0^t \int_0^\tau \G_{1-2u}[N^2\phi](\hat\tau)\left((\tau-\hat\tau)(t-\tau)\right)^{-1/2}d\hat\tau \de\tau = o(t^{3/2}),\qquad\text{as }t\to 0.
\end{equation}
In addition, an application of Theorem \ref{t:limit_u} and the dominated convergence theorem implies that
\begin{equation}
\label{eq:asymp2} 
\int_0^t\int_{\partial\Omega}\left(1-2u(\tau,\cdot)\right)(4\deltasr-N^2)\phi \de\area(t-\tau)^{1/2}d\tau=o(t^{3/2}) \qquad\text{as }t\to 0.
\end{equation}
Thus, using \eqref{eq:asymp1} and \eqref{eq:asymp2}, we can improve \eqref{eq:diff_exp_aux}, obtaining
\begin{equation}
\label{eq:comparison_diff1}
I\phi(t,0)-\op\phi(t,0)= 2\G_{1-2u}[\phi](t)+\frac{1}{2}\int_{\partial\Omega}N\phi \de\area \, t+o(t^{3/2}).
\end{equation}
On the other hand, the quantity $I\phi(t,0)-\op\phi(t,0)$ has a complete asymptotic series by Proposition \ref{prop:asym_rel_tmp}, which at order $3$ becomes:
\begin{equation}
\label{eq:comparison_diff2}
I\phi(t,0)-\op\phi(t,0)=\int_{\partial\Omega}g(\nabla\phi,\nabla\deltas)d\area \, t+o(t^{3/2}),\qquad\text{as }t\to0.
\end{equation}
Comparing \eqref{eq:comparison_diff1} and \eqref{eq:comparison_diff2}, we deduce that, as $t\to0$,
\begin{align}
2\G_{1-2u}[\phi](t) &=-\frac{1}{2}\int_{\partial\Omega}N\phi \de\area \, t+o(t^{3/2})+\int_{\partial\Omega}g(\nabla\phi,\nabla\deltas)d\area \, t+o(t^{3/2})\\
										&=-\frac{1}{2}\int_{\partial\Omega}\phi\Delta\deltas\de\area \, t+o(t^{3/2}).
\end{align}
Finally, using the linearity of the functional $\G_v[\phi]$ with respect to $v$, we conclude the proof.
\end{proof}

\begin{rmk}
\label{rmk:limitation}
The asymptotics \eqref{eq:asymp_G} for the functional $\G_u[\phi](t)$ is the best result that we are able to achieve. In the expression \eqref{eq:diff_exp_aux}, the problematic term is given by \eqref{eq:asymp2}, i.e.
\begin{equation}
\int_0^t\int_{\partial\Omega}\left(1-2u(\tau,\cdot)\right)(4\deltasr-N^2)\phi \de\area(t-\tau)^{1/2}d\tau,
\end{equation}
which can not be expressed in terms of $\G_u$, hence the only relevant information is given by Theorem \ref{t:limit_u}. In conclusion, we can not repeat the strategy of the proof of Theorem \ref{t:asymp_G}, replacing the series of $\G_u$ at order $3$ in \eqref{eq:diff_exp_aux} to deduce the higher-order terms.
\end{rmk}

\subsection{Fourth-order asymptotics}
In this section we prove Theorem \ref{t:intro1}. We recall here the statement.

\begin{thm}
\label{t:fourth_ord}
Let $M$ be a compact \sr manifold, equipped with a smooth measure $\omega$, and let $\Omega\subset M$ be an open subset whose boundary is smooth and has no characteristic points. Then, as $t\to0$,
\begin{equation}
\label{eq:fourth_ord}
H_\Omega(t) =\omega(\Omega)-\frac{1}{\sqrt\pi}\area(\partial\Omega) t^{1/2}-\frac{1}{12\sqrt\pi}\int_{\partial\Omega}\left(2g(\nabla\deltas,\nabla(\Delta\deltas ))-(\Delta\deltas)^2 \right)d\area \, t^{3/2}+o(t^{2}).
\end{equation}
\end{thm}

Before giving the proof of the theorem, let us comment on its strategy. Recall that, on the one hand, for a cutoff function $\phi\in\extfunsp$ which is identically $1$ close to $\partial\Omega$, cf.\ \eqref{eq:cutoff}, the localization principle \eqref{eq:true_loc_H} holds, namely
\begin{equation}
\label{eq:loc_aux}
\omega(\Omega)-H_\Omega(t)=I\phi(t,0)+O(t^\infty),\qquad\text{as }t\to0.
\end{equation} 
Moreover, by the iterated Duhamel's principle for $I\phi(t,0)$, cf.\ Lemma \ref{lem:aux_exp2}, we can deduce expression \eqref{eq:expr_I_G}, namely
\begin{equation}
\label{eq:expr_I_aux}
I\phi(t,0) =2\G_{1-u}[\phi](t)+\frac{1}{\sqrt\pi}\int_0^t\G_{1-u}[N\phi](t)d\area (t-\tau)^{-1/2} d\tau+O(t^{3/2}).
\end{equation} 
On the other hand, we have an asymptotic series of the functional $\G_u$ at order $3$, cf.\ Theorem \ref{t:limit_u}. Therefore, if we naively insert this series in \eqref{eq:expr_I_aux}, we can obtain, at most, a third-order asymptotic expansion of the relative heat content $H_\Omega(t)$, whereas we are interested in the fourth-order expansion. 

Using the outside contribution, we are able to overcome this difficulty. In particular, applying Proposition \ref{prop:asym_rel_tmp}, for a function $\phi\in\extfunsp$ which is identically $1$ close to $\partial\Omega$, we have the following asymptotic relation:
\begin{equation}
\label{eq:diff_aux}
I\phi(t,0)=\op\phi(t,0)+O(t^\infty),\qquad\text{as }t\to0.
\end{equation}
Notice that \eqref{eq:diff_aux} is a direct consequence of Proposition \ref{prop:asym_rel_tmp} since all the coefficients of the expansion vanish. Therefore, thanks to \eqref{eq:diff_aux}, we can rephrase \eqref{eq:loc_aux} as follows:
\begin{equation}
\label{eq:H_sum}
\omega(\Omega)-H_\Omega(t)=\frac{1}{2}\left(I\phi(t,0)+\op\phi(t,0)\right)+O(t^\infty),\qquad\text{as }t\to0.
\end{equation}
The advantage of \eqref{eq:H_sum} is that we can now apply the iterated Dirichlet principle for the sum $I\phi+\op\phi$, cf.\ Lemma \ref{lem:formal_exp_sum}. Already at order $3$, we obtain
\begin{equation}
\label{eq:expr_sum4_aux}
\left(I\phi+\op\phi\right)(t,0)=\frac{2}{\sqrt\pi}\int_{\partial\Omega}\phi \de\area \, t^{1/2}+\frac{1}{\sqrt\pi}\int_0^t \G_{1-2u}[N\phi](\tau)(t-\tau)^{-1/2}d\tau+O(t^{3/2}),
\end{equation}
where $N$ is the operator defined in \eqref{eq:sr_N_true}. As we can see, in \eqref{eq:expr_sum4_aux}, the functional $\G_{u}$ occurs for the first time in the second iteration of the Duhamel's principle, as opposed to the expansion for $I\phi$, where it appeared already in the first application, cf.\ \eqref{eq:expr_I_aux}. Hence we gain an order with respect to the asymptotic series of $\G_u$. More generally, if we were able to develop the $k$-th order asymptotics for $\G_u$, this would imply the $(k+1)$-th order expansion for $H_\Omega(t)$.

\begin{proof}[Proof of Theorem \ref{t:fourth_ord}]
Following the discussion above, it is enough to expand the sum $I\phi+\op\phi$, with $\phi\in\extfunsp$. For this quantity, Lemma \ref{lem:formal_exp_sum} holds, namely we have the following iterated version of Duhamel's principle:
\begin{flalign}
\left(I\phi+\op\phi\right)(t,0) &=\frac{2}{\sqrt\pi}\int_{\partial\Omega}\phi \de\area \, t^{1/2}+\frac{1}{\sqrt\pi}\int_0^t \G_{1-2u}[N\phi](\tau)(t-\tau)^{-1/2}d\tau \label{eq:expr_sum4}\\
						 &\quad +\frac{1}{6\sqrt\pi}\int_{\partial\Omega}(4\deltasr+N^2)\phi \de\area \, t^{3/2}\notag\\
						 &\quad + \frac{1}{4\pi^{3/2}}\int_0^t\!\int_0^\tau\! \int_0^{\hat\tau}\! \G_{1-2u}[N^3\phi](s)\left((\hat\tau-s)(\tau-\hat\tau)(t-\tau)\right)^{-1/2}\! ds\de\hat\tau \de\tau\notag\\
						 &\quad +\frac{1}{4\sqrt\pi}\int_0^t \G_{1-2u}[(6N\Delta-N^3-2\Delta N)\phi](\tau)(t-\tau)^{1/2}d\tau+O(t^{5/2}),\notag
\end{flalign}
where $N$ is defined in \eqref{eq:sr_N_true}. Moreover, recall that by Theorem \ref{t:asymp_G}, the functional $\G_{1-2u}[\phi]$ has the following expansion for any $\phi\in\extfunsp$:
\begin{equation}
\label{eq:asymp_G_aux}
\G_{1-2u}[\phi](t)=-\frac{1}{4}\int_{\partial\Omega}\phi\Delta\deltas\de\area \, t+o(t^{3/2}),\qquad\text{as }t\to0.
\end{equation}
Thus, replacing the term $\G_{1-2u}[N\phi]$ in \eqref{eq:expr_sum4} with the expansion \eqref{eq:asymp_G_aux} for $N\phi\in\extfunsp$, we obtain the following asymptotic as $t\to 0$:
\begin{multline}
\label{eq:sum_aux}
I\phi(t,0)+\op\phi(t,0)=\frac{2}{\sqrt\pi}\int_{\partial\Omega}\phi \de\area \, t^{1/2}-\frac{1}{3\sqrt\pi}\left(\int_{\partial\Omega}N\phi\Delta\deltas\de\sigma\right)t^{3/2}\\
+\frac{1}{6\sqrt\pi}\int_{\partial\Omega}(4\deltasr+N^2)\phi \de\area \, t^{3/2}+o(t^2),
\end{multline}
for any $\phi\in\extfunsp$. In particular, if we choose $\phi\in\extfunsp$ such that $\phi\equiv 1$ close to $\partial\Omega$, then on the one hand, from \eqref{eq:sum_aux}, we obtain, as $t\to0$,
\begin{multline}
I\phi(t,0)+\op\phi(t,0) =\frac{2}{\sqrt\pi}\sigma(\partial\Omega)t^{1/2}\\
+\frac{1}{6\sqrt\pi}\int_{\partial\Omega}\left(2g(\nabla\deltas,\nabla(\Delta\deltas ))-(\Delta\deltas)^2 \right)d\area \, t^{3/2}+o(t^2).
\end{multline}
On the other hand, the asymptotic relation \eqref{eq:H_sum} holds. This concludes the proof.
\end{proof}

\paragraph{Third-order vs fourth-order asymptotics.}
We stress that we could have obtained the third-order asymptotic expansion of $H_\Omega(t)$ \emph{without} introducing the sum of the inside and outside contributions $I\phi+\op\phi$, and only using the Duhamel's principle for $I\phi$, cf.\ Lemma \ref{lem:aux_exp2}, and the asymptotic series for $\G_u$, cf.\ Theorem \ref{t:asymp_G}. However, for the improvement to the fourth-order asymptotics, the argument of the sum of contributions seems necessary.


\subsection{The weighted relative heat content}
\label{sec:weighted}

Adapting the proof of Theorem \ref{t:fourth_ord}, one can prove a slightly more general result which we state here for completeness.

\begin{prop}
\label{prop:4th_exp_H}
Let $M$ be a compact \sr manifold, equipped with a smooth measure $\omega$, and let $\Omega\subset M$ be an open subset whose boundary is smooth and has no characteristic points. Let $\chi\in C_c^\infty(M)$ and define the \emph{weighted relative heat content}
\begin{equation}
H_\Omega^\chi(t)=\int_\Omega u(t,x)\chi(x)d\omega(x), \qquad \forall\,t>0.
\end{equation}
Then, as $t\to0$, 
\begin{equation}
\label{eq:4th_exp_H}
\begin{split}
H_\Omega^\chi(t) &=\int_\Omega\chi \de\omega-\frac{1}{\sqrt\pi}\int_{\partial\Omega}\chi \de\area \, t^{1/2}-\frac{1}{2}\int_{\partial\Omega}g(\nabla\chi,\nabla\deltas)d\area \, t\\
								 &\quad-\left(\frac{1}{12\sqrt\pi}\int_{\partial\Omega}(4\deltasr+N^2)\chi \de\area-\frac{1}{6\sqrt\pi}\int_{\partial\Omega}(N\chi)\Delta\deltas\de\area\right)t^{3/2}\\
								 &\quad-\frac{1}{2}\int_{\partial\Omega}g(\nabla(\Delta\chi),\nabla\deltas)d\area \, t^2+o(t^2).
\end{split}
\end{equation}
\end{prop}

\begin{proof}
Let us consider a cutoff function $\phi$ as in \eqref{eq:cutoff}. Then, applying the usual localization argument, cf.\ \eqref{eq:split}, we have:
\begin{equation}
\int_\Omega\chi(x)d\omega(x)-H_\Omega^\chi(t)=I[\phi\chi](t,0)+O(t^\infty),\qquad\text{as }t\to0,
\end{equation} 
where now the function $\phi\chi\in\extfunsp$ and $\phi\chi=\chi$ close to $\partial\Omega$. 

As we did in the proof of Theorem \ref{t:fourth_ord}, we relate $H_\Omega^\chi(t)$ with the sum of contributions. Applying Proposition \ref{prop:asym_rel_tmp}, we have the following asymptotic relation at order $4$: 
\begin{equation}
I[\phi\chi](t,0)-\op[\phi\chi](t,0)=\int_{\partial\Omega}g(\nabla\chi,\nabla\deltas)d\area \, t+\int_{\partial\Omega}g(\nabla(\Delta\chi),\nabla\deltas)d\area \, t^2+o(t^2),	
\end{equation} 
as $t\to0$, having used the fact that $\phi\chi\equiv\chi$ close to $\partial\Omega$. Notice that this relation coincides with \eqref{eq:diff_aux} when $\chi\equiv 1$ close to $\partial\Omega$. Thus, we obtain
\begin{multline}
\int_\Omega\chi(x)d\omega(x)-H_\Omega^\chi(t)=\frac{1}{2}\left(I[\phi\chi](t,0)+\op[\phi\chi](t,0)\right)\\
+\int_{\partial\Omega}g(\nabla\chi,\nabla\deltas)d\area \, t+\int_{\partial\Omega}g(\nabla(\Delta\chi),\nabla\deltas)d\area \, t^2+o(t^2),\qquad\text{as }t\to0.
\end{multline}
Finally, applying \eqref{eq:sum_aux} for $I[\phi\chi](t,0)+\op[\phi\chi](t,0)$, we conclude.
\end{proof}

\begin{rmk}
We compare the coefficients of the expansions of $H_\Omega(t)$ and $Q_\Omega(t)$, defined in \eqref{eq:def_chc}, respectively. On the one hand, by \cite[Thm.\ 5.8]{MR4223354}, the $k$-th coefficient of the expansion of $Q_\Omega(t)$ is of the form
\begin{equation}
-\int_{\partial\Omega}D_k(\chi)d\sigma, \qquad\forall\chi\in C_c^\infty(M),
\end{equation}
where $D_k$ is a differential operator acting on  $C_c^\infty(M)$ and belonging to $\mathrm{span}_\R\{\Delta,N\}$ as algebra of operators. On the other hand, Proposition \ref{prop:4th_exp_H} shows that this is no longer true for the third coefficient of the expansion of $H_\Omega(t)$, as we need to add the operator multiplication by $\Delta\delta$.
\end{rmk}

\section{An alternative approach using the heat kernel asymptotics}
\label{sec:asym_u_bd}
As we can see by a first application of Duhamel's principle, cf.\ \eqref{eq:1st_exp_I}, and its iterations, the small-time asymptotics of $u(t,\cdot)\rvert_{\partial\Omega}$, together with uniform estimates on the remainder with respect to $x\in\partial\Omega$, would be enough to determine the asymptotic expansion of the relative heat content, at any order. 

In Theorem \ref{t:limit_u}, we studied the zero-order asymptotics of $u(t,\cdot)\rvert_{\partial\Omega}$. The technique used for its proof does not work at higher-order, since the exponential remainder term in \eqref{eq:ineq_lim} would be unbounded as $t\to0$. In this section, we comment how such an higher-order asymptotics of $u(t,\cdot)\rvert_{\partial\Omega}$ can be obtained exploiting the asymptotic formula for the heat kernel proved in \cite[Thm.\ A]{YHT-2}. 

Let $M$ be a compact \sr manifold and $\Omega\subset M$ an open subset with smooth boundary. For $x\in\partial\Omega$, let us consider $\psi=(z_1,\ldots,z_n)\colon U\rightarrow V$ a chart of privileged coordinates centered at $x$, with $U$ a relatively compact set. Since the heat kernel is exponentially decaying outside the diagonal, cf.\ \eqref{eq:js_est}, 
\begin{equation}
\label{eq:u_loc}
\begin{split}
u(t,x)&=\int_\Omega p_t(x,y)d\omega(y)=\int_{\Omega\cap U}p_t(x,y)d\omega(y)+O(t^\infty)\\&=\int_{V_1}p_t(0,z)d\omega(z)+O(t^\infty),
\end{split}
\end{equation}
where $V_1=\psi(U\cap\Omega)$, and we denote with the same symbols $\omega$ and $p_t(0,z)$ the coordinate expression of the measure and heat kernel, respectively. For example, if $x\in \partial\Omega$ is non-characteristic, we may choose $\psi$ as in \eqref{eq:bd_priv}, then $V_1=V\cap\{z_1>0\}$. Recall the asymptotic expansion of the heat kernel of Theorem \ref{t:hk_exp}, evaluated in $(0,z)$: for any $m\in\N$ and compact set $K\subset (0,\infty)\times V$,  
\begin{equation}
\label{eq:hk_exp_sec}
|\eps|^\mathcal{Q}p_{\eps^2\tau}(0,\delta_\eps(z))=\hat p_\tau(0,z)+\sum_{i=0}^m \eps^i f_i(\tau,0,z)+o(|\eps|^m),\qquad\text{as }\eps\to 0,
\end{equation}
uniformly as $(\tau,z)\in K$, where $\mathcal{Q}$, $\hat p$ and $f_i$'s are defined in Section \ref{sec:prel}. We will omit the dependance on the center of the privileged coordinates $x$, it being fixed for the moment. At this point, we would like to integrate \eqref{eq:hk_exp_sec} to get information of $u(t,x)$ as $t\to 0$. Proceeding formally, let us choose the parameters $\eps,\tau$ in \eqref{eq:hk_exp_sec} such that:
\begin{equation}
\label{eq:par}
\eps^2\tau=t,\qquad \eps=t^{\frac{\alpha}{2\alpha+1}},\qquad\tau=t^{\frac{1}{2\alpha+1}},
\end{equation}
for some $\alpha>0$ to be fixed. For convenience of notation, set 
\begin{equation}
V_s=\delta_s(V_1)\qquad\forall\,s\in [-1,1],
\end{equation}
then, split the integral over $V_1$ in \eqref{eq:u_loc} in two, so that the first one is computed on $V_\eps$ and the second one is computed on its complement in $V_1$, i.e. $V_1\setminus V_\eps$. Notice that, by usual off-diagonal estimates, see \cite[Prop.\ 3]{JS-estimates} and our choice of the parameter $\eps$ as in \eqref{eq:par}, the following is a remainder term, independently of the value of $\alpha$:
\begin{equation}
\int_{V_1\setminus V_\eps}p_t(0,z)d\omega(z)=O\left(e^{-\beta\frac{\eps^2}{t}}\right)=O(t^\infty),\qquad\text{as }t\to0.
\end{equation} 
Thus, writing the measure in coordinates $d\omega(z)=\omega(z)dz$ with $\omega(\cdot)\in C^\infty(V_1)$, we have, as $t\to0$,
\begin{flalign}
u(t,x) &=\int_{V_\eps}p_t(0,z)\omega(z)dz+O(t^\infty)=\int_{V_1}\eps^Qp_{\eps^2\tau}(0,\delta_\eps(z))\omega(\delta_\eps(z))dz+O(t^\infty)\\
						 &=\int_{V_1}\left(\hat p_\tau(0,z)+\sum_{i=0}^{m-1} \eps^i f_i(\tau,0,z)+\eps^mR_m(\eps,\tau,z)\right)\omega(\delta_\eps(z))dz+O(t^\infty),\qquad\label{eq:u_loc_exp}
\end{flalign}
where $R_m$ is a smooth function on $[-1,1]\times (0,\infty)\times \R^n$, such that
\begin{equation}
\label{eq:Rm_est}
\sup_{\substack{\eps\in[-1,1] \\ z\in K}}\left|R_m(\eps,\tau,z)\right|\leq C_m(\tau,K),
\end{equation}
for any compact set $K\subset\R^n$, according to \eqref{eq:hk_exp_sec}. Up to restricting the domain of privileged coordinates $U$, we can assume that \eqref{eq:Rm_est} holds on $\overline V$. By our choices \eqref{eq:par}, we would like the following term
\begin{equation}
\label{eq:remainder}
t^{\frac{m\alpha}{2\alpha+1}}\int_{V_1}\left|R_m\left(t^{\frac{\alpha}{2\alpha+1}},t^{\frac{1}{2\alpha+1}},z\right)\right|\omega(\delta_{t^{\alpha/(2\alpha+1)}}(z))dz
\end{equation}
to be an error term of order greater than $\frac{m-1}{2}$, as $t\to0$. Thus, assume for the moment that $\forall K\subset V$ compact and $\forall m\in\N$, $\exists \ell=\ell(m,K)\in\N$ and $C_m(K)>0$ such that
\begin{equation}
\label{eq:pol_dec}
\sup_{\substack{\eps\in[-1,1] \\ z\in K}}\left|R_m(\eps,\tau,z)\right|\leq \frac{C_m(K)}{\tau^\ell},\qquad\forall\,\tau\in (0,1).\tag{$\mathbf{H}$}
\end{equation}
Thanks to assumption \eqref{eq:pol_dec}, choosing $\alpha$ large enough, we see that \eqref{eq:remainder} is a $o(t^{\frac{m-1}{2}})$. Performing the change of variables $z\mapsto\delta_{1/\sqrt{\tau}}(z)$ in \eqref{eq:u_loc_exp}, and exploiting the homogeneity properties of $\hat p$ and $f_i$, namely \eqref{eq:hom_prop}, we finally obtain the following expression for $u$ as $t\to 0$:
\begin{equation}
\label{eq:final_exp}
u(t,x)=\int_{V_{t^{-1/(2(2\alpha+1))}}}\left(\hat p_1(0,z)+\sum_{i=0}^{m-1} t^{i/2} a_i(z)\right)\omega(\delta_{\sqrt{t}}(z))dz+o(t^{\frac{m-1}{2}}),
\end{equation}
having set $a_i(z)=f_i(1,0,z)$, for all $i\in\N$. Therefore, we find an asymptotic expansion of $u(t,x)$ under assumption \eqref{eq:pol_dec}, which is crucial to overcome the fact that \eqref{eq:hk_exp_sec} is formulated on an asymptotic neighborhood of the diagonal, and not uniformly as $\tau\to 0$. It is likely\footnote{Personal communication by Yves Colin de Verdi\`ere, Luc Hillairet and Emmanuel Tr\'elat.} that \eqref{eq:pol_dec} can be proven in the nilpotent case, and more generally when the ambient manifold is $M=\R^n$ and the generating family of the \sr structure, $\{X_1,\ldots,X_N\}$ satisfies the uniform H\"ormander polynomial condition, see \cite[App.\ B]{YHT-2} for details. Although this strategy could be used to prove the existence of an asymptotic expansion of $H_\Omega(t)$, we refrain to go in this direction since two technical difficulties would arise nonetheless:
\begin{itemize}
\item \emph{Uniformity of the expansion of $u(t,x)$ with respect to $x\in\partial\Omega$}. In the non-equiregular case, cf.\ Section \ref{sec:nilpotent} for details, the expansion \eqref{eq:hk_exp_sec} is not uniform as $x$ varies in compact subsets of $M$, hence the same would be true for the expansion \eqref{eq:final_exp}.
\item \emph{Computations of the coefficients}. The coefficients appearing in \eqref{eq:final_exp} depend on the nilpotent approximation at $x\in\partial\Omega$ and are not clearly related to the invariants of $\partial\Omega$. 
\end{itemize}

Our strategy avoids almost completely the knowledge of the small-time asymptotics of $u(t,\cdot)_{\partial\Omega}$, it being based on an asymptotic series of the auxiliary functional $\G_u$. Moreover, we stress that our method to prove the asymptotics of $H_\Omega(t)$ up to order $4$, cf.\ Theorem \ref{t:intro1}, holds for any \sr manifold, including also the non-equiregular ones.

\begin{rmk}
In order to pass from \eqref{eq:final_exp} to the asymptotic expansion of $H_\Omega(t)$, we would use Duhamel's formula, which holds under the non-characteristic assumption. This means that, even though \eqref{eq:hk_exp_sec} of course is true even in presence of characteristic points, we can't say much about the asymptotics of $H_\Omega(t)$ in the general case.
\end{rmk}

\section{The non-compact case}
\label{sec:loc_noncomp}

In the non-compact case, we have the following difficulties:
\begin{itemize}
\item The localization principle, cf.\ Proposition \ref{prop:easy_loc}, may fail. 
\item Set $u(t,x)=e^{t\Delta}\mathds{1}_{\Omega}(x)$ and $u^c(t,x)=e^{t\Delta}\mathds{1}_{\Omega^c}(x)$. If the manifold is not stochastically complete, the relation $u(t,x)+u^c(t,x) = 1$ does not hold.
\item The Gaussian bounds for the heat kernel and its time-derivatives, \`a la Jerison and Sanchez-Calle \cite[Thm.\ 3]{JS-estimates}, may not hold, thus Lemma \ref{lem:hp_idp_aux} may not be true.
\end{itemize}
 %

\begin{defn}
\label{def:loc_space}
Let $M$ be a \sr manifold, equipped with a smooth measure $\omega$. We say that $(M,\omega)$ is \emph{(globally) doubling} if there exist constants $C_D>0$ such that:
\begin{equation}
V(x,2\rho)\leq C_D V(x,\rho),\qquad\forall \,\rho>0, \ x\in M,
\end{equation}
where $V(x,\rho)=\omega(B_{\rho}(x))$. We say that $(M,\omega)$ satisfies a \emph{(global) weak Poincar\'e inequality}, if there exist constants $C_P>0$ such that,
\begin{equation}
\int_{B_\rho(x)}\left|f-f_{x,\rho}\right|^2d\omega\leq C_P\rho^2\int_{B_{2\rho}(x)}\|\nabla f\|^2d\omega,\qquad \rho>0,\ x\in M,
\end{equation}
for any smooth function $f\in C^\infty(M)$. Here $f_{x,\rho}=\frac{1}{V(x,\rho)}\int_{B_\rho(x)}fd\omega$.
We refer to these properties as local whenever they hold for any $\rho<\rho_0$.
\end{defn} 

\begin{rmk}
\label{rmk:stoc_compl}
If $M$ is a \sr manifold, equipped with a smooth globally doubling measure $\omega$, then it is stochastically complete, namely
\begin{equation}
\label{eq:stoc_compl}
\int_Mp_t(x,y)d\omega(y)=1,\qquad\forall\,t>0,\ x\in M.
\end{equation}
This is a straightforward consequence of the characterization given by \cite[Thm.\ 4]{MR1301456} on the volume growth of balls. 
\end{rmk}

\begin{thm}
\label{thm:doub_poincare}
Let $M$ be a complete \sr manifold, equipped with a smooth measure $\omega$. Assume that $(M,\omega)$ is globally doubling and satisfies a global weak Poincar\'e inequality. Then, there exist constants $C_k, c_k>0$, for any integer $k\geq 0$, depending only on $C_D,C_P$, such that, for any $x,y\in M$ and $t>0$,
\begin{equation}
\label{eq:offdiag_der}
|\partial_t^kp_t(x,y)|\leq \frac{C_kt^{-k}}{V(x,\sqrt{t})}\exp\left(-\frac{d_\mathrm{SR}^2(x,y)}{c_k t}\right),
\end{equation}
where we recall $V(x,\sqrt{t})=\omega(B_{\sqrt{t}}(x))$.

In addition, there exists constants $C_\ell,c_\ell>0$, depending only on $C_D,C_P$, such that, for any $x,y\in M$ and $t>0$,
\begin{equation}
\label{eq:lower_bound}
p_t(x,y)\geq \frac{C_\ell}{V(x,\sqrt{t})}\exp\left(-\frac{d_\mathrm{SR}^2(x,y)}{c_\ell t}\right).
\end{equation}
\end{thm}

\begin{proof}
Define the \sr Hamiltonian as the smooth function $H : T^*M \to \R$,
\begin{equation}
H(\lambda) = \frac{1}{2}\sum_{i=1}^N \langle \lambda, X_i \rangle^2, \qquad \lambda \in T^*M,
\end{equation}
where $\{X_1,\ldots,X_N\}$ is a generating family for the sub-Riemannian structure. Then, following the notations of \cite{MR1387522}, one can easily verify that 
\begin{equation}
\mathcal{E}(u,v)=\int_M 2H(du,dv)d\omega,\qquad\forall\, u,v\in C_c^\infty(M),
\end{equation}
where $H$ is the \sr Hamiltonian viewed as a bilinear form on fibers, defines a strongly local Dirichlet form with domain $\mathrm{dom}(\mathcal{E})=C_c^\infty(M)$. Notice that the Friedrichs extension of $\mathcal{E}$ is exactly the sub-Laplacian, moreover, the intrinsic metric 
\begin{equation}
d_I(x,y)=\sup\{|u(x)-u(y)| \text{ s.t. }u\in C_c^\infty(M),\, |2H(du,du)|\leq 1\},\qquad\forall x,y\in M. 
\end{equation}
coincides with the usual \sr distance, as $|2H(du,du)|=\|\nabla u\|^2$, cf.\\ \cite[Ch.\ 2, Prop.\ 12.4]{MR3470142}. Thus, $\mathcal{E}$ is also strongly regular and, by our assumptions on $(M,\omega)$, \cite[Thm.\ 4.3]{MR1150597} holds true, proving \eqref{eq:offdiag_der}. 
For the Gaussian lower bound \eqref{eq:lower_bound}, it is enough to apply \cite[Cor.\ 4.10]{MR1387522}, cf.\ also \cite[Thm.\ 4.2]{MR1150597}. This concludes the proof.  
\end{proof}

\begin{rmk}
Theorem \ref{thm:doub_poincare} ensures that the time-derivatives of the heat kernel satisfy Gaussian bounds, which are sufficient to prove Lemma \ref{lem:hp_idp_aux} in the non-compact case. This lemma is crucial to obtain the asymptotics expansion of $H_\Omega(t)$ at order \emph{strictly greater} than $1$.
\end{rmk}

We prove now the non-compact analogue of Proposition \ref{prop:easy_loc}. 
 
\begin{cor}
\label{cor:noncomp_loc}
Under the assumptions of Theorem \ref{thm:doub_poincare}, let $\Omega\subset M$ be an open subset with smooth boundary. Then, for any $K\subset M$ closed subset of $M$ such that $K\cap\partial\Omega=\emptyset$, we have:
\begin{equation}
\mathds{1}_\Omega(x)-u(t,x)=O(t^\infty),\qquad\text{as }t\to 0,\quad\text{uniformly for }x\in K,
\end{equation}
where $u(t,x)=e^{t\Delta}\mathds{1}_\Omega(x)$ is the solution to \eqref{eq:cauchy_prob}.
\end{cor}

\begin{proof}
Let us assume that $K\subset \Omega$ such that $K\cap\partial\Omega=\emptyset$. The other part of the statement can be done similarly. 

Since $M$ is stochastically complete, cf.\ Remark \ref{rmk:stoc_compl}, for any $x\in K$, we can write:
\begin{equation}
\mathds{1}_\Omega(x)-u(t,x)=1-e^{t\Delta}\mathds{1}_\Omega(x)=e^{t\Delta}1(x)-e^{t\Delta}\mathds{1}_\Omega(x)=\int_{M\setminus\Omega}p_t(x,y)d\omega(y).
\end{equation}
Thanks to Theorem \ref{thm:doub_poincare}, we can apply \eqref{eq:offdiag_der} for $k=0$ obtaining
\begin{equation}
\int_{M\setminus\Omega}p_t(x,y)d\omega(y)\leq \int_{M\setminus\Omega}\frac{C_0}{V(x,\sqrt{t})}\exp\left(-\frac{d_\mathrm{SR}^2(x,y)}{c_0 t}\right)d\omega(y),
\end{equation}
for suitable constants $C_0,c_0>0$ not depending on $x,y\in M$, $t>0$. Now, fix $L>1$: since $K\subset \Omega$ is closed with empty intersection with $\partial\Omega$, and thus well-separated from $\partial\Omega$, we deduce there exists $a=a(K)>0$ such that $d_\mathrm{SR}(x,y)>a$ for any $x\in K,\, y\in M\setminus\Omega$, and so 
\begin{align}
\int_{M\setminus\Omega}p_t(x,y)d\omega(y) &\leq \int_{M\setminus\Omega}\frac{C_0}{V(x,\sqrt{t})}\exp\left(-\frac{d_\mathrm{SR}^2(x,y)}{c_0 t}\right)d\omega(y)\\
																					&\leq \exp\left(-\frac{C(a,L)}{c_0t}\right)\int_{M\setminus\Omega}\frac{C_0}{V(x,\sqrt{t})}\exp\left(-\frac{d_\mathrm{SR}^2(x,y)}{2^Lc_0 t}\right)d\omega(y),\qquad\label{eq:exponential_int}
\end{align}
where $C(a,L)=\frac{a^2(2^L-1)}{2^L}>0$. Thus, if we prove that the integral in \eqref{eq:exponential_int} is finite, we conclude. Firstly, recall the Gaussian lower bound \eqref{eq:lower_bound}, which holds thanks to Theorem \ref{thm:doub_poincare}:
\begin{equation}
\label{eq:lower_bound_aux}
p_t(x,y)\geq \frac{C_\ell}{V(x,\sqrt{t})}\exp\left(-\frac{d_\mathrm{SR}^2(x,y)}{c_\ell t}\right).
\end{equation}
for suitable constants constants $C_\ell,c_\ell>0$, not depending on $x,y\in M$, $t>0$. Secondly, by the doubling property of $\omega$, it is well-known that there exists $C_D',s>0$ depending only on $C_D$ such that
\begin{equation}  
\label{eq:doub_equiv}
V(x,R)\leq C_D'\left(\frac{R}{\rho}\right)^sV(x,\rho),\qquad\forall\, \rho\leq R.
\end{equation}
Therefore, choosing $L>1$ so big that $\tilde c^2=(2^Lc_0)/c_\ell>1$ and applying \eqref{eq:doub_equiv} for $\rho=\sqrt{t}$ and $R=\tilde c \sqrt{t}$, we have $R>\rho$ and
\begin{equation}  
\label{eq:doub_equiv2}
V\left(x,\tilde c\sqrt{t}\right)\leq \tilde C V(x,\sqrt{t}),\qquad\forall\, t>0,
\end{equation}
having denoted by $\tilde C=C_D'\tilde c^s>0$. Finally, using \eqref{eq:doub_equiv2} and the Gaussian lower bound \eqref{eq:lower_bound_aux}, we can estimate the integral in \eqref{eq:exponential_int} as follows:
\begin{multline}
\int_{M\setminus\Omega}\frac{1}{V(x,\sqrt{t})}\exp\left(-\frac{d_\mathrm{SR}^2(x,y)}{2^Lc_0 t}\right)d\omega(y) \\
																							\leq\int_M\frac{\tilde C}{V(x,\tilde c\sqrt{t})}\exp\left(-\frac{d_\mathrm{SR}^2(x,y)}{c_\ell\tilde{c}t}\right)d\omega(y)
																					\leq\frac{\tilde C}{C_\ell}\int_M p_{\tilde t}(x,y)d\omega(y)
																							\leq\frac{\tilde C}{C_\ell},
\end{multline}
where $\tilde t=\tilde c t$. Since the resulting constant does not depend on $x\in K$, we conclude the proof. 
\end{proof}

Using Corollary \ref{cor:noncomp_loc} and adopting the same strategy of the compact case, one can finally prove the following result. 
\begin{thm}
\label{thm:final_result}
Let $M$ be a complete \sr manifold, equipped with a smooth measure $\omega$. Assume that $(M,\omega)$ is globally doubling and satisfies a global weak Poincar\'e inequality. Let $\Omega\subset M$ be an open and bounded subset whose boundary is smooth and has no characteristic points. Then, as $t\to0$,
\begin{equation}
H_\Omega(t) =\omega(\Omega)-\frac{1}{\sqrt\pi}\area(\partial\Omega) t^{1/2}-\frac{1}{12\sqrt\pi}\int_{\partial\Omega}\left(2g(\nabla\deltas,\nabla(\Delta\deltas ))-(\Delta\deltas)^2 \right)d\area \, t^{3/2}+o(t^{2}).
\end{equation} 
\end{thm}

\begin{rmk}
Theorem \ref{thm:final_result} holds true also for the weighted relative heat content, cf.\ Section \ref{sec:weighted}. In both cases, we do not know whether its assumptions are sharp in the non-compact case. 
\end{rmk}

\subsection{Notable examples}
\label{sec:examples}
We list here some notable examples of \sr manifolds satisfying the assumptions of Theorem \ref{thm:doub_poincare}. For these examples Theorem \ref{thm:final_result} is valid.

\begin{itemize}
\item $M$ is a Lie group with polynomial volume growth, the distribution is generated by a family of left-invariant vector fields satisfying the H\"{o}rmander condition and $\omega$ is the Haar measure. This family includes also Carnot groups. See for example \cite{MR1435484,MR1150597,MR3008437}.
\item $M=\R^n$, equipped with a \sr structure induced by a family of vector fields $\{Y_1,\ldots,Y_N\}$ with bounded coefficients together with their derivatives, and satisfying the H\"{o}rmander condition. Under these assumptions, the Lebesgue measure is doubling, cf.\ \cite[Thm.\ 1]{MR793239}, and the Poincar\'e inequality is verified in \cite{MR850547}. We remark that these works provide the local properties of Definition \ref{def:loc_space}, with constants depending only on the $C^k$-norms of the vector fields $Y_i$, for $i=1,\ldots,N$. Thus, if the $C_k$-norms are globally bounded, we obtain the corresponding global properties.  
\item $M$ is a complete Riemannian manifold with metric $g$, equipped with the Riemannian measure, and with non-negative Ricci curvature.
\end{itemize}
We mention that a Riemannian manifold $M$ with Ricci curvature bounded below by a negative constant satisfies only locally Definition \ref{def:loc_space}, i.e. for some $\rho_0<\infty$, depending on the Ricci bound. Nevertheless, we can prove Corollary \ref{cor:noncomp_loc} in this case, as Li and Yau provides an upper Gaussian bound, see \cite[Cor.\ 3.1]{MR834612}, and a lower bound as \eqref{eq:lower_bound} holds, cf. \ \cite[Cor.\ 2]{MR1681640}. Thus, the first-order asymptotic expansion of $H_\Omega(t)$, cf.\ Theorem \ref{thm:1st_exp_H}, is valid in this setting.

\appendix

\section{Iterated Duhamel's principle for \texorpdfstring{$I_\Omega\phi(t,0)$}{IOmegaphi(t,0)}}\label{app:computations}
In this section, we study the iterated Duhamel's principle for the $I_\Omega\phi$, cf.\ Definition \ref{def:ILambda_om}. The main result is Lemma \ref{lem:aux_exp2}, which will imply formulas \eqref{eq:expr_I}, \eqref{eq:diff_exp_aux} and \eqref{eq:expr_sum4}.

The next proposition is a version of the iterated Duhamel's principle taken from \cite[Prop.\ A.1]{MR4223354}, which we recall here. 
\begin{prop}
\label{prop:iter_duhamel}
Let $F\in C^\infty((0,\infty)\times[0,+\infty))$ be a smooth function compactly supported in the second variable and let $L=\partial_t-\partial_{r}^2$. Assume that the following conditions hold:
\begin{itemize}
\item[(i)] $\displaystyle L^kF(0,r)=\lim_{t\to 0}L^kF(t,r)$ exists in the sense of distributions\footnote{Namely, for any $\psi\in C^\infty([0,\infty))$, there exists finite $\lim_{t\to 0}\int_0^\infty f(t,r)\psi(r)dr$. With a slight abuse of notation, we define $\int_0^\infty f(0,r)\psi(r)dr=\lim_{t\to 0}\int_0^\infty f(t,r)\psi(r)dr$.} for any $k\geq 0$; 
\item[(ii)] $L^kF(t,0)$ and $\partial_rL^kF(t,0)$ converge to a finite limit as $t\to 0$, for any $k\geq 0$.
\end{itemize}
Then, for all $m\in\N$ and $t>0$, we have 
\begin{multline}
\label{eq:iter_duhamel}
F(t,0)=\sum_{k=0}^m\left(\frac{t^k}{k!}\int_0^\infty e(t,r,0)L^kF(0,r)dr-\frac{1}{\sqrt{\pi}k!}\int_0^t\partial_rL^kF(\tau,0)(t-\tau)^{k-1/2}d\tau\right)
\\
+\frac{1}{m!}\int_0^t\int_0^\infty e(t-\tau,r,0)L^{m+1}F(\tau,r)(t-\tau)^mdr \de\tau,\qquad \
\end{multline}
where $e(t,r,s)$ is the Neumann heat kernel on the half-line, cf.\ \eqref{eq:neumanheat}.
\end{prop} 

We want to apply Proposition \ref{prop:iter_duhamel} to the function $I_\Omega\phi(t,0)$, thus, we study in detail the operators $L^kI_\Omega$, for any $k\geq 1$. Define iteratively the family of matrices of operators, acting on smooth functions:
\begin{equation}
M_{kj}=\begin{pmatrix}
Q_{kj}&S_{kj}\\
P_{kj}&R_{kj}
\end{pmatrix},
\end{equation}
as follows. Set
\begin{equation}
M_{10}=
\begin{pmatrix}
\deltasr&-\deltasr N_\Omega\\
N_\Omega&-N_\Omega^2+\deltasr
\end{pmatrix}
\qquad\text{and}\qquad
M_{11}=
\begin{pmatrix}
0&N_\Omega\\
0&0
\end{pmatrix},
\end{equation}
and, for all $k\geq 1$ and $0\leq j \leq k$, set
\begin{equation}
\label{eq:matrix_iter}
M_{kj}=M_{10}M_{k-1,j}+M_{11}M_{k-1,j-1},
\end{equation}
while $M_{kj}=0$, for all other values of the indices, i.e.\ $k<0$, $j<0$ or $k<j$. Here $N_\Omega$ is the operator defined in \eqref{eq:sr_N}, namely
\begin{equation}
\label{eq:sr_N_rem}
N_\Omega\phi= 2g\left(\nabla\phi,\nu\right)+\phi\,\diverg(\nu), \qquad\,\forall\,\phi\in \extfunsp,
\end{equation}
where $\nu$ is the inward-pointing normal from $\Omega$. 

Recall the definition of $I_\Omega$ and $\Lambda_\Omega$: for any $\phi\in\extfunsp$ and for all $t> 0, r\geq 0$, 
\begin{align}
I_\Omega\phi(t,r) &= \int_{\omNice{r}}{\left(1- u(t,x)\right)\phi(x)d\omega(x)},\\
\Lambda_\Omega\phi(t,r) &=-\partial_rI_\Omega\phi(t,r)=-\int_{\bdomNice{r}}{\left(1- u(t,y)\right)\phi(y)d\sigma(y)},
\end{align} 
where $u(t,\cdot)=e^{t\Delta}\mathds{1}_\Omega(\cdot)$. Iterations of $L^kI_\Omega\phi$ satisfy the following lemma. 

\begin{lem}
\label{lem:iter_LI_om}
Let $M$ be a sub-Riemannian manifold, equipped with a smooth measure $\omega$, and let $\Omega \subset M$ be an open relatively compact subset whose boundary is smooth and has no characteristic points. Then, as operators on $C^\infty_c(\Omega_{-\smallpar}^{\smallpar})$, we have:
\begin{itemize}
\item[(i)] $LI_\Omega=I_\Omega\deltasr+\Lambda_\Omega N_\Omega$;
\item[(ii)] $L\Lambda_\Omega=\Lambda_\Omega\left(-N_\Omega^2+\deltasr\right)+\partial_tI_\Omega N_\Omega-I_\Omega\deltasr N_\Omega$;
\item[(iii)] For any $k\in\N$, 
\begin{equation}
\label{eqn:iter_expr}
L^kI_\Omega=\sum_{j=0}^{k}{\frac{\partial^j}{\partial t^j}(\Lambda_\Omega P_{kj}+I_\Omega Q_{kj})}\qquad\text{and}\qquad L^k\Lambda_\Omega=\sum_{j=0}^{k}{\frac{\partial^j}{\partial t^j}(\Lambda_\Omega R_{kj}+I_\Omega S_{kj})}.
\end{equation}
\end{itemize}
Here we mean that, for any $\phi\in\extfunsp$, the operator $L^k$ acts on the functions $I_\Omega\phi(t,r)$, $\Lambda_\Omega\phi(t,r)$. Analogously the right-hand side when evaluated in $\phi$ is a function of $(t,r)$.
\end{lem}

\begin{proof}
The proof of items $(i)$ and $(ii)$ follows from Proposition \ref{prop:mean_val} and the divergence theorem, cf.\ \cite[Lem.\ A.2]{MR4223354}. We show how to recover the iterative law \eqref{eq:matrix_iter}. 

Consider the vector $V=\left(I_\Omega,\Lambda_\Omega\right)$, then by items $(i)$ and $(ii)$, we have 
\begin{equation}
\label{eq:1st_step}
LV=\left(LI_\Omega,L\Lambda_\Omega\right)=VM_{10}+\partial_tVM_{11}.
\end{equation}
Notice that the operator $L^k$ contains at most $k$ derivatives with respect to $t$, therefore we have
\begin{equation}
L^kV=\sum_{j=0}^k\partial_t^j\left(V M_{kj}\right),\qquad\forall\,k\geq 0,
\end{equation}
On the other hand, we can evaluate $L^kV$, using \eqref{eq:1st_step}:
\begin{align}
L^kV &=L\left(L^{k-1}V\right)=\sum_{j=0}^{k-1}L\partial_t^j\left(V M_{k-1,j}\right)=\sum_{j=0}^{k-1}\partial_t^j\left(LV M_{k-1,j}\right)\\
		 &=\sum_{j=0}^{k-1}\partial_t^jVM_{10}M_{k-1,j}+\sum_{j=0}^{k-1}\partial_t^{j+1}V M_{11}M_{k-1,j}.
\end{align}
Reorganizing the sum, we find \eqref{eq:matrix_iter}, concluding the proof. 
\end{proof}
 
We want to apply Proposition \ref{prop:iter_duhamel} to $I_\Omega\phi(t,r)$ for $k\geq 2$, in order to obtain higher-order asymptotics. However, Lemma \ref{lem:iter_LI_om} shows that $L^kI_\Omega$, for $k\geq 2$, involves time derivatives of $u(t,x)$ which are not well-defined at $\partial\Omega$ as $t\to 0$. Therefore, we consider the following approximation of $I_\Omega\phi$ and $\Lambda_\Omega\phi$, respectively: fix $\epsilon>0$ and define, for any $t>0, r\geq0$,
\begin{align}
I_\epsilon\phi(t,r)=&= \int_{\omNice{r}}\left(1-u_\epsilon(t,x)\right)\phi(x)d\omega(x), \label{eq:aux_ops1}
\\
\Lambda_\epsilon\phi(t,r) &= -\partial_r I_\epsilon\phi(t,r)=\int_{\bdomNice{r}}\left(1-u_\epsilon(t,x)\right)\phi(y)d\area(y), \label{eq:aux_ops2}
\end{align} 
where $u_\epsilon(t,x)=e^{t\Delta}\mathds{1}_{\omNice{\epsilon}}(x)$. We recall that, for any $a\in\R$, $\omNice{a} = \{x\in M\mid \deltas(x) >a\}$. Notice that, by the dominated convergence theorem, we have
\begin{equation}
I_\epsilon\phi(t,0)\xrightarrow{\epsilon\to 0}I_\Omega\phi(t,0),\qquad\text{uniformly on }[0,T],
\end{equation} 
and, in addition, Lemma \ref{lem:iter_LI_om} holds unchanged also for $I_\epsilon$ and $\Lambda_\epsilon$.

\begin{lem}
\label{lem:hp_idp_aux}
Let $M$ be a compact sub-Riemannian manifold, equipped with a smooth measure $\omega$, and let $\Omega \subset M$ be an open subset whose boundary is smooth and has no characteristic points. Let $\psi\in C^\infty([0,\infty))$, $\epsilon\in (0,\smallpar)$ and define
\begin{equation}
\psi^{(-1)}(r)=\int_0^r\psi(s)ds,\qquad\forall r\geq 0.
\end{equation}
Then, for any $\phi\in C^\infty_c(\Omega_{-\smallpar}^{\smallpar})$, the following identities hold:
\begin{itemize}
\item[(i)] $\displaystyle\lim_{t\to 0}\int_0^\infty{\frac{\partial^j}{\partial t^j}\Lambda_\epsilon\phi(t,r) \psi(r)dr}=\begin{cases}\displaystyle \int_{\omprimeNice{\epsilon}}{\phi(\psi\circ\deltas)d\omega} & \text{ if }j=0,\\[10pt]
\displaystyle -\int_{\omNice{\epsilon}}{\deltasr^j(\phi(\psi\circ\deltas))d\omega}& \text{ if }j\geq 1;\end{cases}$
\item[(ii)] $\displaystyle\lim_{t\to 0}\int_0^\infty{\frac{\partial^j}{\partial t^j}I_\epsilon\phi(t,r)\psi(r)dr}=\begin{cases}\displaystyle \int_{\omprimeNice{\epsilon}}{\phi\left(\psi^{(-1)}\circ\deltas\right) d\omega}& \text{ if }j=0,\\[10pt]
\displaystyle -\int_{\omNice{\epsilon}}{\deltasr^j\left(\phi\left(\psi^{(-1)}\circ\deltas\right)\right) d\omega}& \text{ if }j\geq 1;\end{cases}$ 
\item[(iii)] $\displaystyle\frac{\partial^j}{\partial t^j}\Lambda_\epsilon\phi(0,0)=\begin{cases}\displaystyle \int_{\partial\Omega}{\phi \de\area}& \text{ if }j=0,\\[10pt]
\displaystyle 0& \text{ if }j\geq 1;\end{cases}$
\item[(iv)] $\displaystyle \frac{\partial^j}{\partial t^j}I_\epsilon\phi(0,0)=\begin{cases}\displaystyle\int_{\omprimeNice{\epsilon}}\phi \de\omega& \text{ if }j=0,\\[10pt]
-\int_{\omNice{\epsilon}}{\deltasr^j\phi \de\omega}& \text{ if }j\geq 1;
\end{cases}$ 
\end{itemize}
where, we recall, $\omNice{\epsilon} = \{x\in M \mid \deltas(x) > \epsilon\}$ and $\omprimeNice{\epsilon}=\Omega\setminus\omNice{\epsilon}$.
\end{lem}

\begin{rmk}
The only difference with respect to \cite[Lem.\ A.4]{MR4223354} is item $(iii)$, which now holds only as $t\to 0$ and not for all positive times. 
\end{rmk}

\begin{proof}[Proof of Lemma \ref{lem:hp_idp_aux}]
We claim that, for any $j\geq 1$, 
\begin{equation}
\label{eq:aux_lim}
\lim_{t\to0}\int_{\Omega}\phi(x)\deltasr^ju_\epsilon(t,x)d\omega(x)=\int_{\omNice{\epsilon}}\deltasr^j\phi(x)d\omega(x).
\end{equation}
Let us prove it by induction: for $j=1$, applying the divergence theorem, we have
\begin{equation}
\label{eq:diverg_j1}
\int_\Omega\phi\deltasr u_\epsilon \de\omega=-\int_{\partial\Omega}\phi g\left(\nabla u_\epsilon,\nabla\deltas\right) d\area+\int_{\partial\Omega}u_\epsilon g\left(\nabla\phi,\nabla\deltas\right) d\area+\int_\Omega u_\epsilon\deltasr\phi \de\omega.
\end{equation}
Let us discuss the first term in \eqref{eq:diverg_j1}: by divergence theorem applied with respect to the set $\Omega^c$, we have  
\begin{equation}
\label{eq:aux_lim2}
\int_{\partial\Omega}\phi g\left(\nabla u_\epsilon,\nabla\deltas\right) d\area=\int_{\Omega^c}\phi\deltasr u_\epsilon \de\omega+\int_{\partial\Omega}u_\epsilon g\left(\nabla \phi,\nabla\deltas\right) d\area-\int_{\Omega^c}u_\epsilon\deltasr\phi \de\omega,
\end{equation}
then, using \cite[Thm.\ 3]{JS-estimates} and noticing that $d_{\mathrm{SR}}(x,y)\geq\epsilon$, for any $x\in\omNice{\epsilon}$ and $y\in\Omega^c$, we conclude that in the limit as $t\to0$, \eqref{eq:aux_lim2} converges to 0. This proves \eqref{eq:aux_lim}, for $j=1$. For $j>1$, proceeding by induction, we conclude. Finally, using the co-area formula \eqref{eq:coarea}, we complete the proof of the statement as in the usual argument of \cite[Lem.\ 5.6]{Savo-heat-cont-asymp}.
\end{proof}

\begin{rmk}
In the non-compact case, under the assumption of Theorem \ref{thm:doub_poincare}, the above lemma holds. In particular, on the one hand, the divergence theorem holds since $\phi$ has compact support. On the other hand, notice that the time derivative estimates \eqref{eq:offdiag_der} are enough to ensure that \eqref{eq:aux_lim2} converges to $0$ as $t\to 0$, regardless of the compactness of the set of integration. The same is true for $j>1$, where higher-order time derivatives appear.  
\end{rmk}

The next step is to apply the iterated Duhamel's principle \eqref{eq:iter_duhamel} to $I_\epsilon$, which now satisfies its assumptions, then, pass to the limit as $\epsilon\to 0$. The computations are long but straightforward: we report here the result at order $t^{5/2}$.

\begin{lem}
\label{lem:aux_exp2}
Under the same assumptions of Lemma \ref{lem:hp_idp_aux}, let $\phi\in C^\infty_c(\Omega_{-\smallpar}^{\smallpar})$. Then, as $t\to 0$, we have:
\begin{align}
I_\Omega\phi(t,0) &=2\G_{1-u}[\phi](t)+\frac{1}{\sqrt\pi}\int_0^t \G_{1-u}[N_\Omega\phi](\tau)(t-\tau)^{-1/2}d\tau
\label{eq:formal_expI}\\
						 &\quad+\frac{1}{2\pi}\int_0^t \int_0^\tau \G_{1-u}[N_\Omega^2\phi](\hat\tau)\left((\tau-\hat\tau)(t-\tau)\right)^{-1/2}d\hat\tau \de\tau\\
						 &\quad+\frac{1}{4\pi^{3/2}}\int_0^t \int_0^\tau \int_0^{\hat\tau} \G_{1-u}[N_\Omega^3\phi](s)\left((\hat\tau-s)(\tau-\hat\tau)(t-\tau)\right)^{-1/2}ds\de\hat\tau \de\tau\\
						 &\quad+\frac{1}{4\sqrt\pi}\int_0^t\int_{\partial\Omega}\left(1-u(\tau,\cdot)\right)(4\deltasr-N_\Omega^2)\phi \de\area(t-\tau)^{1/2}d\tau\\
						 &\quad+\frac{1}{4\sqrt\pi}\int_0^t \G_{1-u}[(6N_\Omega\Delta-N_\Omega^3-2\Delta N_\Omega)\phi](\tau)(t-\tau)^{1/2}d\tau+O(t^{5/2}),
\end{align}
where $u(t,\cdot)=e^{t\Delta}\mathds{1}_{\Omega}$ and $\G_u[\phi]$ is the functional defined in \eqref{eq:def_G}. We recall that $N_\Omega$ is the operator defined in \eqref{eq:sr_N_rem}.
\end{lem}
%

The expression \eqref{eq:expr_I} is a direct consequence of \ref{lem:aux_exp2}. Moreover, we can apply it, when the base set is chosen to be $\Omega^c$. Then, evaluating the difference between $I_\Omega\phi(t,0)$ and $I_{\Omega^c}\phi(t,0)$ we obtain the asymptotic equality \eqref{eq:diff_exp_aux}, which is proved in the next lemma. We use the shorthands $I$, $I^c$ for $I_\Omega$ and $I_{\Omega^c}$ respectively. 

\begin{lem}
\label{lem:formal_exp_diff}
Under the same assumptions of Lemma \ref{lem:hp_idp_aux}, let $\phi\in C^\infty_c(\Omega_{-\smallpar}^{\smallpar})$. Then, as $t\to 0$, we have:
\begin{flalign}
\left(I\phi-\op\phi\right)(t,0) &= 2\G_{1-2u}[\phi](t)+\frac{1}{2}\int_{\partial\Omega}N\phi \de\area \, t\\
									&\quad+\frac{1}{2\pi}\int_0^t \int_0^\tau \G_{1-2u}[N^2\phi](\hat\tau)\left((\tau-\hat\tau)(t-\tau)\right)^{-1/2}d\hat\tau \de\tau\\
									&\quad+\frac{1}{4\sqrt\pi}\int_0^t\int_{\partial\Omega}\left(1-2u(\tau,\cdot)\right)(4\deltasr-N^2)\phi \de\area(t-\tau)^{1/2}d\tau+O(t^{2}),
\end{flalign}
where $N$ is the operator given by
\begin{equation}
N\phi=2g(\nabla\phi,\nabla\deltas)+\phi\Delta\deltas,\qquad\forall\,\phi\in \extfunsp,
\end{equation}
with $\delta\colon M\rightarrow \R$ the signed distance function from $\partial\Omega$. 
\end{lem}

\begin{proof}
Firstly, we apply Lemma \ref{lem:aux_exp2} to $I\phi$: we obtain exactly the expression \eqref{eq:formal_expI}, with the operator $N_\Omega=N$. Secondly, for the outside contribution, recall that we have the following equality of smooth functions:
\begin{equation}
1-u^c(t,x)=1-e^{t\Delta}\mathds{1}_{\Omega^c}(x)=e^{t\Delta}\mathds{1}_\Omega(x)=u(t,x),\qquad\forall\,t>0,\ x\in M.
\end{equation}
Therefore, when we apply Lemma \ref{lem:aux_exp2} to $\op\phi$, we replace $1-u^c(t,\cdot)=1-e^{t\Delta}\mathds{1}_{\Omega^c}$ with the function $u(t,\cdot)=e^{t\Delta}\mathds{1}_\Omega(\cdot)$. Moreover, the operator $N_{\Omega^c}$ defined in \eqref{eq:sr_N_rem}, for the set $\Omega^c$, is equal to $-N$, since the inward-pointing normal to $\Omega^c$ is $-\nabla\deltas$. Therefore, writing the difference of the two contributions, and noticing that $\Omega$ and its complement share the boundary, we have:
\begin{flalign}
\left(I\phi-\op\phi\right)(t,0) &= 2\G_{1-2u}[\phi](t)+\frac{1}{\sqrt\pi}\int_0^t \G_{1}[N\phi](\tau)(t-\tau)^{-1/2}d\tau\\
									&\quad+\frac{1}{2\pi}\int_0^t\! \int_0^\tau \G_{1-2u}[N^2\phi](\hat\tau)\left((\tau-\hat\tau)(t-\tau)\right)^{-1/2}d\hat\tau \de\tau\\
									&\quad+\frac{1}{4\pi^{3/2}}\int_0^t\! \int_0^\tau\! \int_0^{\hat\tau} \G_{1}[N^3\phi](s)\left((\hat\tau-s)(\tau-\hat\tau)(t-\tau)\right)^{-1/2} ds\de\hat\tau \de\tau\qquad\label{eq:rem_term_diff1}\\
									&\quad+\frac{1}{4\sqrt\pi}\int_0^t\!\int_{\partial\Omega}\left(1-2u(\tau,\cdot)\right)(4\deltasr-N^2)\phi \de\area(t-\tau)^{1/2}d\tau\\
									&\quad+\frac{1}{4\sqrt\pi}\int_0^t \G_{1}[(6N\Delta-N^3-2\Delta N)\phi](\tau)(t-\tau)^{1/2}d\tau+O(t^{5/2}).\label{eq:rem_term_diff2}
\end{flalign}
To conclude, it is enough to notice that the functional $\G_{1}$ can be explicitly computed:
\begin{equation}
\G_{1}[\phi](t)=\frac{1}{\sqrt\pi}\int_{\partial\Omega}\phi \de\area \, t^{1/2},\qquad\forall\,\phi\in\extfunsp.
\end{equation}
Thus, the terms in \eqref{eq:rem_term_diff1} and \eqref{eq:rem_term_diff2} are remainder of order $O(t^2)$.
\end{proof}

Applying Lemma \ref{lem:aux_exp2} to the sum of $I_\Omega\phi(t,0)$ and $I_{\Omega^c}\phi(t,0)$ instead, we obtain \eqref{eq:expr_sum4}. The proof of this result is not provided here, as it is similar to the proof of Lemma \ref{lem:formal_exp_diff}.

\begin{lem}
\label{lem:formal_exp_sum}
Under the same assumptions of Lemma \ref{lem:hp_idp_aux}, let $\phi\in C^\infty_c(\Omega_{-\smallpar}^{\smallpar})$. Then, as $t\to 0$, we have:
\begin{flalign}
\left(I\phi+\op\phi\right)(t,0) &=\frac{2}{\sqrt\pi}\int_{\partial\Omega}\phi \de\area \, t^{1/2}+\frac{1}{\sqrt\pi}\int_0^t \G_{1-2u}[N\phi](\tau)(t-\tau)^{-1/2}d\tau \label{eq:expr_sum4_old}\\
						 &\quad +\frac{1}{6\sqrt\pi}\int_{\partial\Omega}(4\deltasr+N^2)\phi \de\area \, t^{3/2}\notag\\
						 &\quad + \frac{1}{4\pi^{3/2}}\int_0^t \int_0^\tau \int_0^{\hat\tau} \G_{1-2u}[N^3\phi](s)\left((\hat\tau-s)(\tau-\hat\tau)(t-\tau)\right)^{-1/2}ds\de\hat\tau \de\tau\notag\\
						 &\quad +\frac{1}{4\sqrt\pi}\int_0^t \G_{1-2u}[(6N\Delta-N^3-2\Delta N)\phi](\tau)(t-\tau)^{1/2}d\tau+O(t^{5/2}).\notag
\end{flalign}
\end{lem}

\end{document}